\documentclass[1 [leqno,11pt]{amsart}
\usepackage{amssymb, amsmath}
\newcommand{\andf}{\quad\hbox{and}\quad}
\newcommand{\with}{\quad\hbox{with}\quad}
\def\Supp{\mathop{\rm Supp}\nolimits\ }
\newcommand{\newcom}{\newcommand}
\def\longformule#1#2{
\displaylines{ \qquad{#1} \hfill\cr \hfill {#2} \qquad\cr } }
\def\inte#1{
\displaystyle\mathop{#1\kern0pt}^\circ }

\newcom{\al}{\alpha}
\newcom{\de}{\delta}
\newcom{\Th}{\Theta}
\newcom{\be}{\beta}
\newcom{\s}{\sigma}
\newcom{\eps}{\epsilon}
\newcom{\ve}{\varepsilon}
\newcom{\ga}{\gamma}
\newcom{\Ga}{\Gamma}
\newcom{\ka}{\kappa}
\newcom{\Lam}{\Lambda}
\newcom{\lam}{\lambda}
\newcom{\vp}{\varphi}
\newcom{\om}{\omega}
\newcom{\Sig}{\Sigma}
\newcom{\sig}{\sigma}
\newcom{\tht}{\theta}
\newcom{\tri}{\triangle}
\newcom{\oo}{\infty}
\newcom{\h}{{\rm h}}
\newcom{\vphi}{\varphi}
\newcom{\cB}{{\mathcal B}}
\newcom{\cC}{{\mathcal C}}
\newcom{\cD}{{\mathcal D}}
\newcom{\cF}{{\mathcal F}}
\newcom{\cL}{{\mathcal L}}
\newcom{\cM}{{\mathcal M}}
\newcom{\cP}{{\mathcal P}}
\newcom{\cS}{{\mathcal S}}
\newcom{\cQ}{{\mathcal Q}}
\newcom{\cT}{{\mathcal T}}
\newcom{\cY}{{\mathcal Y}}
\newcom{\cZ}{{\mathcal Z}}
\newcom{\R}{\Bbb R}
\newcom{\T}{\Bbb T}
\newcom{\N}{\Bbb N}
\newcom{\Z}{\Bbb Z}
\newcom{\C}{\Bbb C}
\newcom{\E}{\Bbb E}
\let\wh=\widehat
\def\dive{\mathop{\rm div}\nolimits}
\let\e=\varepsilon
\def\ue{u^\varepsilon}
\def\ve{v^\varepsilon}
\def\up{u_{\Psi}}
\def\vp{v_{\Psi}}
\def\ekt{e^{\frak{K}t}}
\def\ektp{e^{\frak{K}t'}}

\newcom{\f}{\frac}
\newcom{\dint}{\displaystyle\int}
\newcom{\dsum}{\displaystyle\sum}
\newcom{\dlim}{\displaystyle\lim}
\newcom{\ov}{\overline}
\newcom{\wt}{\widetilde}
\newcom{\pa}{\partial}
\newcom{\p}{\partial}
\newcom\na{\nabla}
\newcom{\D}{\Delta}
\newcom\rto{\rightarrow}
\newcom\lto{\leftarrow}
\newcom\mto{\mapsto}
\newcom{\disp}{\displaystyle}
\newcom{\non}{\nonumber}
\newcom{\no}{\noindent}
\newcom{\QED}{$\square$}

\def\eqdefa{\buildrel\hbox{\footnotesize def}\over =}

\newcommand{\beq}{\begin{equation}}
\newcommand{\eeq}{\end{equation}}
\newcommand{\ben}{\begin{eqnarray}}
\newcommand{\een}{\end{eqnarray}}
\newcommand{\beno}{\begin{eqnarray*}}
\newcommand{\eeno}{\end{eqnarray*}}

\newtheorem{Def}{Definition}[section]
\newtheorem{thm}{Theorem}[section]
\newtheorem{lem}{Lemma}[section]
\newtheorem{rmk}{Remark}[section]

\newtheorem{prop}{Proposition}[section]
\renewcommand{\theequation}{\thesection.\arabic{equation}}

\begin{document}
\title[Hydrostatic approximation of the Navier-Stokes equations]
{On the hydrostatic approximation of the Navier-Stokes equations in a thin strip}

\author{Marius Paicu}
\address{Universit\'e  Bordeaux \\
 Institut de Math\'ematiques de Bordeaux\\
F-33405 Talence Cedex, France} \email{marius.paicu@math.u-bordeaux.fr}

\author{Ping Zhang}%
\address{Academy of
Mathematics $\&$ Systems Science and  Hua Loo-Keng Key Laboratory of
Mathematics, The Chinese Academy of Sciences, Beijing 100190, China, and School of Mathematical Sciences,
University of Chinese Academy of Sciences, Beijing 100049, China.
} \email{zp@amss.ac.cn}

\author{Zhifei Zhang}
\address{School of  Mathematical Science, Peking University, Beijing 100871,
P. R. CHINA} \email{zfzhang@math.pku.edu.cn}

\date{\today}

\begin{abstract}
In this paper, we first prove the global well-posedness of a scaled anisotropic  Navier-Stokes system  and the hydrostatic Navier-Stokes system in a 2-D striped  domain with small analytic data in the tangential variable. Then we justify the limit from the anisotropic Navier-Stokes system to the hydrostatic Navier-Stokes system
with analytic data.
\end{abstract}
\maketitle


\qquad\qquad Radius of analyticity.


\renewcommand{\theequation}{\thesection.\arabic{equation}}
\setcounter{equation}{0}
\section{Introduction}\label{sect1}

This paper is concerned with the study of the Navier-Stokes system in a thin-striped domain  and
 the hydrostatic approximation of these equations when the depth of the domain and the viscosity  converge to zero
  simultaneously in a related way.  This is a classical model in geophysical fluid dynamics where  the vertical dimension of the domain is very small compared with the  horizontal  dimension of the domain. In this case, the viscosity is not isotropic and we have to use the anisotropic Navier-Stokes system with a ``turbulent" viscosity.  The formal limit thus obtained is the hydrostatic Navier-Stokes equations which are currently used as a standard model to describes the atmospheric flows and also oceanic flows in oceanography (see \cite{Pe, Pi}).

\smallskip

When we consider Dirichlet boundary conditions on the top and the  bottom of a 2-D striped domain, we are able to prove the global well-posedness
of both the anisotropic Navier-Stokes system and  the hydrostatic/Prandtl approximate equations when the initial data is small and analytic in the tangential variable. This should be regarded as a global Cauchy-Kowalevskaya theorem for small analytic data, which originates from \cite{Ch04}. The proof of this type of
 results requires the  control of the loss of the radius of the analyticity of the solution. Taking the advantage of the Poincar\'e inequality in the the strip, we are able to control the analyticity of the solution globally in time.  We  also rigorously  prove the convergence of the anisotropic Navier-Stokes system  to the hydrostatic/Prandtl equations in the natural framework of the analytic data in the tangential variable. We  now present a precise description of the problem that we shall investigate.

 \smallskip

We consider two-dimensional incompressible Navier-Stokes equations in a thin strip: $\cS^\e\eqdefa \bigl\{(x,y)\in\R^2:\ 0<y<\e \ \bigr\},$
\begin{equation}\label{NS}
 \quad\left\{\begin{array}{l}
\displaystyle \p_tU+U\cdot\na U-\e^2\D U+\na P=0\quad \ \mbox{in}\ \ \cS^\e\times ]0,\infty[,\\
\displaystyle \dive U=0,
\end{array}\right.
\end{equation}
where $U(t,x,y)$ denotes the velocity of the fluid and $P(t,x,y)$ denotes the scalar pressure function which guarantees the divergence free condition of the velocity field $U$.
We complement the system \eqref{NS} with the non-slip boundary condition
\beno
U|_{y=0}=U|_{y=\e}=0,
\eeno
and the initial condition
\beno
U|_{t=0}=\left(u_0\bigl(x,\f{y}\e\bigr), \e v_0\bigl(x,\f{y}\e\bigr) \right)=U_0^\e \quad  \mbox{in} \ \ \cS^\e.
\eeno

As in \cite{BL, Lions96}, we  write
\beq \label{S1eq3}
U(t,x,y)=\Bigl(u^\e\bigl(t,x,\f{y}\e\bigr), \e v^\e\bigl(t,x,\f{y}\e\bigr) \Bigr) \andf P(t,x,y)=p^\e\bigl(t,x,\f{y}\e\bigr).
\eeq
Let $\cS\eqdefa \bigl\{(x,y)\in\R^2:\ 0<y<1\bigr\}$.
Then the system \eqref{NS} becomes the following scaled anisotropic Navier-Stokes system:
\begin{equation}\label{S1eq4}
 \quad\left\{\begin{array}{l}
\displaystyle \p_tu^\e+u^\e\p_x u^\e+v^\e\p_yu^\e-\e^2\p_x^2u^\e-\p_y^2u^\e+\p_xp^\e=0\ \ \mbox{in} \ \cS\times ]0,\infty[,\\
\displaystyle \e^2\left(\p_tv^\e+u^\e\p_x v^\e+v^\e\p_yv^\e-\e^2\p_x^2v^\e-\p_y^2v^\e\right)+\p_yp^\e=0,\\
\displaystyle \p_xu^\e+\p_yv^\e=0,\\
\displaystyle \left(u^\e, v^\e\right)|_{t=0}=\left(u_0, v_0\right),
\end{array}\right.
\end{equation}
together with the boundary condition
\beq \label{S1eq5}
\left(u^\e, v^\e\right)|_{y=0}=\left(u^\e, v^\e\right)|_{y=1}=0.
\eeq

Formally taking $\e\to 0$ in the system \eqref{S1eq4}, we obtain the hydrostatic Navier-Stokes/Prandtl equations:
\begin{equation}\label{S1eq6}
 \quad\left\{\begin{array}{l}
\displaystyle \p_tu+u\p_x u+v\p_yu-\p_y^2u+\p_xp=0\ \ \mbox{in} \ \cS\times ]0,\infty[,\\
\displaystyle\p_yp=0\\
\displaystyle \p_xu+\p_yv=0,\\
\displaystyle u|_{t=0}=u_0,
\end{array}\right.
\end{equation}
together with the boundary condition
\beq \label{S1eq7}
\left(u, v\right)|_{y=0}=\left(u, v\right)|_{y=1}=0.
\eeq

The goal of this paper is to justify the limit from the system \eqref{S1eq4} to the system \eqref{S1eq6}. The first step is to establish the well-posedness of the two system. Similar to the Prandtl equation, the nonlinear term $v\p_yu$ in \eqref{S1eq6} will lead to one derivative loss in the $x$ variable in the process of energy estimates. Thus, it is natural to work with analytic data in order to overcome this difficulty if we don't impose extra structural assumptions on the initial data \cite{GD, Renard09}. Indeed,  for the data which is analytic in $x,y$ variables,
Sammartino and Caflisch \cite{Caf} established the local
well-posedness result of \eqref{S1eq6} in the upper half space. Later, the analyticity in $y$
variable was removed by Lombardo, Cannone and Sammartino in
\cite{Can}. The main argument used in  \cite{Caf, Can} is to apply
the abstract Cauchy-Kowalewskaya (CK) theorem. We also mention a
 well-posedness result of Prandtl system for a class of data with
Gevrey regularity \cite{GM}. Lately,
 for a class of convex data, G\'erard-Varet, Masmoudi and Vicol \cite{GMV} proved the well-posedness of the  system \eqref{S1eq6}  in the Gevrey class.

\smallskip

Now let us state our main results.\smallskip

 The first result  is the global well-posedness of the system \eqref{S1eq4} with small analytic data in $x$ variable. The main interesting point is that the smallness of data is independent of $\e$ and there holds the global uniform estimate \eqref{S3eq18} with respect to the parameter $\e$.

\begin{thm}\label{thm1.1}
{\sl Let $a>0.$ We assume that the initial data satisfies
\beq\label{S1eq8}
\bigl\|e^{a|D_x|}(u_0,\e v_0)\bigr\|_{\cB^{\f12}}\leq c_0a
\eeq
for some $c_0$ sufficiently small. Then the system \eqref{S1eq4} has a unique global solution $(u,v)$ so that
\beq \label{S3eq18}
\begin{split}
\|\ekt(\up^\e,\e\vp^\e)\|_{\wt{L}^\infty(\R^+;\cB^{\f12})}
+&\|\ekt\p_y(\up^\e,\e\vp^\e)\|_{\wt{L}^2(\R^+;\cB^{\f12})}\\
+&\e^2\bigl\|\ekt(\up^\e,\e\vp^\e)\|_{\wt{L}^2(\R^+;\cB^{\f32})} \leq C\bigl\|e^{a |D_x|}(u_0,\e v_0)\bigr\|_{\cB^{\f12}},
\end{split}
\eeq
where $(\up^\e,\vp^\e)$ will be given by \eqref{S3eq1} and the constant $\frak{K}$ is determined by Poincar\'e inequality on the strip $\cS$ (see \eqref{S3eq5a}), and the functional spaces will be presented in Section \ref{sect2}. }
\end{thm}

The second result is the global well-posedness of the hydrostatic Navier-Stokes system \eqref{S1eq6} with small analytic data in $x$ variable.  We remark that similar global result seems open for the Prandtl equation, where only a lower bound of the lifespan to the solution was obtained (see \cite{ZZ}).

\begin{thm}\label{th1.2}
{\sl Let $a>0.$ We assume that the initial data satisfies
\beq\label{S1eq8a}
\bigl\|e^{a|D_x|}u_0\bigr\|_{\cB^{\f12}}\leq c_1a
\eeq
for some $c_1$ sufficiently small and there holds the compatibility condition $\pa_x\int_0^1u_0dy=0$. Then the system (\ref{S1eq6}) has a  unique global solution $u$ so that
\beq\label{S1eq12}
\|\ekt u_\Phi\|_{\wt{L}^\infty(\R^+;\cB^{\f12})}+\|\ektp\p_y u_\Phi\|_{\wt{L}^2(\R^+;\cB^{\f12})} \leq C\bigl\|e^{a |D_x|}u_0\bigr\|_{\cB^{\f12}},
\eeq where $u_\Phi$ will be determined by \eqref{S4eq1}.
Furthermore, if $e^{a |D_x|}u_0\in \cB^\f52, e^{a |D_x|}\pa_yu_0\in \cB^\f32$ and
\beq\label{S1eq14}
\bigl\|e^{a|D_x|}u_0\bigr\|_{\cB^{\frac12}}\leq \frac{c_2a}{1+\bigl\|e^{a|D_x|}u_0\bigr\|_{\cB^{\frac32}}}
\eeq
for some $c_2$ sufficiently small,  then exists a positive constant $C$ so that for $\lam=C^2\bigl(1+\bigl\|e^{a|D_x|}u_0\bigr\|_{\cB^{\f32}}\bigr)$ and $1\le s\le \f52$, one has
\beq \label{S1eq13}
\begin{split}
&\|\ekt u_\Phi\|_{\wt{L}^\infty(\R^+;\cB^{s})}+\|\ekt\p_y u_\Phi\|_{\wt{L}^2(\R^+;\cB^{s})} \leq C\bigl\|e^{a|D_x|}u_0\bigr\|_{\cB^{s}},\\
&\|\ekt(\pa_tu)_\Phi\|_{\wt{L}^2(\R^+;\cB^{\f32})}+\|\ekt\p_y^2u_\Phi\|_{\wt{L}^2(\R^+;\cB^{\f32})} \leq C\big(\bigl\|e^{a|D_x|}\pa_yu_0\bigr\|_{\cB^{\f32}}+\bigl\|e^{a|D_x|}u_0\bigr\|_{\cB^{\f52}}\big).
\end{split}
\eeq}
\end{thm}

The third result is concerning the convergence from the scaled anisotropic  Navier-Stokes system \eqref{S1eq4}  to the hydrostatic Navier-Stokes system \eqref{S1eq6}.

\begin{thm}\label{thm3}
{\sl Let $a>0$ and $(u_0^\e, v_0^\e)$ satisfy \eqref{S1eq8}. Let $u_0$ satisfy $e^{a|D_x|}u_0\in \cB^\f12\cap \cB^\f52, e^{a|D_x|}\pa_yu_0\in \cB^\f32,$ and  there holds  \eqref{S1eq14}
for some $c_2$ sufficiently small and the compatibility condition
$\pa_x\int_0^1u_0dy=0.$ Then we have
\beq\label{S1eq15}
\begin{split}
\|(w^1_\Th,\e w^2_\Th)\|_{\wt{L}^\infty_t(\cB^{\f12})}
+\|\p_y(w^1_\Th,&\e w^2_\Th)\|_{\wt{L}^2_t(\cB^{\f12})}+\e\|(w^1_\Th,\e w^2_\Th)\|_{\wt{L}^2_t(\cB^{\f32})}\\
&\leq C\Bigl(\bigl\|e^{a|D_x|}(u_0^\e-u_0,\e (v_0^\e-v_0))\bigr\|_{\cB^{\f12}}+
M\e\Bigr).
\end{split}
\eeq
Here $w^1\eqdefa u^\e-u,\ w^2\eqdefa v^\e-v$ and $v_0$ is determined from $u_0$ via $\p_xu_0+\p_yv_0=0$ and $v_0|_{y=0}=v_0|_{y=1}=0,$
and $(w^1_\Th,\e w^2_\Th)$ will be given by \eqref{theta}.
}
\end{thm}

We remark that without the smallness conditions \eqref{S1eq8} and \eqref{S1eq14}, we can prove the convergence of the system
\eqref{S1eq4} to the system \eqref{S1eq6} on a fixed time interval $[0,T].$

\medbreak
We end this introduction by the  notations that  will be used in all that follows. For~$a\lesssim b$, we mean that there is a
uniform constant $C,$ which may be different on different lines,
such that $a\leq Cb$.  We denote by $(a|b)_{L^2}$ the $L^2(\cS)$
inner product of $a$ and $b$.
We designate
by $L^p_T(L^q_{\rm h}(L^r_{\rm v}))$ the space $L^p(]0,T[;
L^q(\R_{x};L^r(\R_y))).$ Finally, we denote by $(d_k)_{k\in\Z}$ (resp. $(d_k(t))_{k\in\Z}$) to be a generic
element of $\ell^1(\Z)$ so that $\sum_{k\in\Z}d_k=1$ (resp. $\sum_{k\in\Z}d_k(t)=1$).

\medskip

\renewcommand{\theequation}{\thesection.\arabic{equation}}
\setcounter{equation}{0}

\section{Littlewood-Paley theory and functional framework}\label{sect2}

In the rest of this paper, we shall frequently  use Littlewood-Paley
decomposition in the horizontal variable $x$. Let us recall from
\cite{BCD} that \beq
\begin{split}
&\Delta_k^{\rm h}a=\cF^{-1}(\varphi(2^{-k}|\xi|)\widehat{a}),\qquad
S^{\rm h}_ka=\cF^{-1}(\chi(2^{-k}|\xi|)\widehat{a}),
\end{split} \label{1.3a}\eeq where $\cF
a$ and $\widehat{a}$  denote the partial  Fourier transform of
the distribution $a$ with respect to $x$ variable,  that is, $
\widehat{a}(\xi,y)=\cF_{x\to\xi}(a)(\xi,y),$
  and $\chi(\tau),$ ~$\varphi(\tau)$ are
smooth functions such that
 \beno
&&\Supp \varphi \subset \Bigl\{\tau \in \R\,/\  \ \frac34 \leq
|\tau| \leq \frac83 \Bigr\}\andf \  \ \forall
 \tau>0\,,\ \sum_{j\in\Z}\varphi(2^{-j}\tau)=1,\\
&&\Supp \chi \subset \Bigl\{\tau \in \R\,/\  \ \ |\tau|  \leq
\frac43 \Bigr\}\quad \ \ \ \andf \  \ \, \chi(\tau)+ \sum_{j\geq
0}\varphi(2^{-j}\tau)=1.
 \eeno

Let us also  recall the functional spaces we are going to use.

\begin{Def}\label{def1.2}
{\sl  Let~$s$ in~$\R$. For~$u$ in~${S}_h'(\cS),$ which
means that $u$ belongs to ~$S'(\cS)$ and
satisfies~$\lim_{k\to-\infty}\|S_k^{\rm h}u\|_{L^\infty}=0,$ we set
$$
\|u\|_{\cB^{s}}\eqdefa\big\|\big(2^{ks}\|\Delta_k^{\rm h}
u\|_{L^{2}}\big)_{k\in\Z}\bigr\|_{\ell ^{1}(\Z)}.
$$
\begin{itemize}

\item
For $s\leq \frac{1}{2}$, we define $ \cB^{s}(\cS)\eqdefa
\big\{u\in{S}_h'(\cS)\;\big|\; \|
u\|_{\cB^{s}}<\infty\big\}.$

\item
If $k$ is  a positive integer and if~$\frac{1}{2}+k< s\leq
\frac{3}{2}+k$, then we define~$ \cB^{s}(\cS)$  as the subset
of distributions $u$ in~${S}_h'(\cS)$ such that
$\p_x^k u$ belongs to~$ \cB^{s-k}(\cS).$
\end{itemize}
}
\end{Def}

In  order to obtain a better description of the regularizing effect
of the diffusion equation, we need to use Chemin-Lerner
type spaces $\widetilde{L}^{\lambda}_T(\cB^{s}(\cS))$.
\begin{Def}\label{def2.2}
{\sl Let $p\in[1,\,+\infty]$ and $T\in]0,\,+\infty]$. We define
$\widetilde{L}^{p}_T(\cB^{s}(\cS))$ as the completion of
$C([0,T]; \,S(\cS))$ by the norm
$$
\|a\|_{\widetilde{L}^{p}_T(\cB^{s})} \eqdefa \sum_{k\in\Z}2^{ks}
\Big(\int_0^T\|\Delta_k^{\rm h}\,a(t) \|_{L^2}^{p}\,
dt\Big)^{\frac{1}{p}}
$$
with the usual change if $p=\infty.$ }
\end{Def}

In order to overcome the difficulty that one can not use Gronwall
type argument in the framework of Chemin-Lerner space, we  need to use the time-weighted
Chemin-Lerner norm, which was introduced by the first two authors in
\cite{PZ1}.

\begin{Def}\label{def1.1} {\sl Let $f(t)\in L^1_{\mbox{loc}}(\R_+)$
be a nonnegative function. We define \beq \label{1.4}
\|a\|_{\wt{L}^p_{t,f}(\cB^{s})}\eqdefa
\sum_{k\in\Z}2^{ks}\Bigl(\int_0^t f(t')\|\D_k^{\rm
h}a(t')\|_{L^2}^p\,dt'\Bigr)^{\f1p}. \eeq}
\end{Def}

 \medbreak
For the convenience of the readers, we recall the following anisotropic
Bernstein type lemma from \cite{CZ1, Pa02}.

\begin{lem} \label{lem:Bern}
 {\sl Let $\cB_{\rm h}$ be a ball
of~$\R_{\rm h}$, and~$\cC_{\rm h}$  a ring of~$\R_{\rm
h}$; let~$1\leq p_2\leq p_1\leq \infty$ and ~$1\leq q\leq \infty.$
Then there holds:

\smallbreak\noindent If the support of~$\wh a$ is included
in~$2^k\cB_{\rm h}$, then
\[
\|\partial_{x}^\alpha a\|_{L^{p_1}_{\rm h}(L^{q}_{\rm v})} \lesssim
2^{k\left(|\al|+\left(\f 1 {p_2}-\f 1 {p_1}\right)\right)}
\|a\|_{L^{p_2}_{\rm h}(L^{q}_{\rm v})}.
\]

\smallbreak\noindent If the support of~$\wh a$ is included
in~$2^k\cC_{\rm h}$, then
\[
\|a\|_{L^{p_1}_{\rm h}(L^{q}_{\rm v})} \lesssim
2^{-kN} \|\partial_{x}^N a\|_{L^{p_1}_{\rm
h}(L^{q}_{\rm v})}.
\]
}
\end{lem}

In the following context, we shall constantly use  Bony's decomposition (see \cite{Bo}) for
the horizontal variable: \ben\label{Bony} fg=T^{\rm h}_fg+T^{\rm
h}_{g}f+R^{\rm h}(f,g), \een where \beno T^{\rm h}_fg\eqdefa \sum_kS^{\rm
h}_{k-1}f\Delta_k^{\rm h}g,\andf R^{\rm
h}(f,g)\eqdefa\sum_k{\Delta}_k^{\rm h}f\widetilde{\Delta}_{k}^{\rm h}g \eeno
with $\widetilde{\Delta}_k^{\rm h}g\eqdefa
\displaystyle\sum_{|k-k'|\le 1}\Delta_{k'}^{\rm h}g$.

\renewcommand{\theequation}{\thesection.\arabic{equation}}
\setcounter{equation}{0}
\section{Global well-posedness of the  system \eqref{S1eq4}}\label{sect1}

In this section, we establish the global well-posedness of the scaled anisotropic
 Navier-Stokes system \eqref{S1eq4} with small analytic data.

\begin{proof}[Proof of Theorem \ref{thm1.1}]
As in \cite{Ch04, CGP, CMSZ, mz1, mz2, ZZ},  for any locally bounded function
$\Psi$ on $\R^+\times \R$, we define \beq\label{S3eq1}
\ue_{\Psi}(t,x,y)\eqdefa\cF_{\xi\to
x}^{-1}\bigl(e^{\Psi(t,\xi)}\widehat{u}^\e(t,\xi,y)\bigr). \eeq We
introduce a key quantity $\eta(t)$ to describe the evolution of
the analytic band of $u^\e:$
\begin{equation}\label{S3eq2}
 \quad\left\{\begin{array}{l}
\displaystyle \dot{\eta}(t)=\e\|\p_x \ue_\Psi(t)\|_{\cB^{\f12}}+\|\p_y\ue_\Psi(t)\|_{\cB^{\f12}},\\
\displaystyle \eta|_{t=0}=0.
\end{array}\right.
\end{equation}
Here  the phase function $\Psi$ is defined by
\beq\label{S3eq3} \Psi(t,\xi)\eqdefa (a-\lam \eta(t))|\xi|. \eeq

In the rest of this section, we shall prove that under the assumption of \eqref{S1eq8}, there holds the {\it a priori} estimate
 \eqref{S3eq18} for smooth enough solutions of \eqref{S1eq4}, and neglect the regularization procedure. For simplicity,
 we shall neglect the script $\e.$ Then in view of \eqref{S1eq4} and \eqref{S3eq1}, we observe that
$(u_\Psi, v_\Psi)$ verifies
\beq\label{S3eq4}
\left\{
\begin{array}{ll}
\p_tu_{\Psi}+\lam\dot{\eta}(t)|D_x|u_{\Psi}+\left(u\p_xu\right)_{\Psi}+\left(v\p_y u\right)_{\Psi}-\e^2\p_x^2 u_{\Psi}
-\p_y^2u_{\Psi}+\p_x p_{\Psi}=0, \\
\e^2\left(\p_tv_{\Psi}+\lam\dot{\eta}(t)|D_x|v_{\Psi}+\left(u\p_xv\right)_{\Psi}+\left(v\p_y v\right)_{\Psi}-\e^2\p_x^2 v_{\Psi}
-\p_y^2v_{\Psi}\right)+\p_y p_{\Psi}=0,\\
\p_x u_{\Psi}+\p_yv_{\Psi} =0\quad\mbox{for}  \quad (t,x,y)\in\R_+\times\cS,\\
\left(u_{\Psi}, v_{\Psi}\right)|_{y=0}=\left(u_\Psi, v_\Psi\right)|_{y=1}=0,
\end{array}
\right.
 \eeq where $|D_x|$
denotes the Fourier multiplier with symbol $|\xi|.$

By applying the dyadic operator
$\D_k^{\rm h}$ to \eqref{S3eq4} and then taking the $L^2$ inner
product of the resulting equation with $\left(\D_k^{\rm
h}u_\Psi, \D_k^\h v_\Psi\right),$  we find \beq \label{S3eq5}
\begin{split}
\f12\f{d}{dt}&\bigl\|\D_k^{\rm h}(u_\Psi, \e v_\Psi)(t)\bigr\|_{L^2}^2+\lam\dot{\eta}\bigl(|D_x|\D_k^{\rm
h}(u_\Psi,\e v_\Psi)\ |\ \D_k^{\rm h}(u_\Psi,\e v_\Psi)\bigr)_{L^2}\\
&+\e^2\bigl\|\p_x\D_k^\h (u_\Psi, \e v_\Psi)\bigr\|_{L^2}^2+\bigl\|\p_y\D_k^\h (u_\Psi, \e v_\Psi)\bigr\|_{L^2}^2\\
=&-\bigl(\D_k^\h\left(u\p_xu\right)_\Psi | \D_k^\h u_\Psi\bigr)_{L^2}-\bigl(\D_k^\h\left(v\p_yu\right)_\Psi | \D_k^\h u_\Psi\bigr)_{L^2}\\
&-\e^2\bigl(\D_k^\h\left(u\p_xv\right)_\Psi | \D_k^\h v_\Psi\bigr)_{L^2}-\e^2\bigl(\D_k^\h\left(v\p_yv\right)_\Psi | \D_k^\h v_\Psi\bigr)_{L^2},
\end{split}
\eeq
where we used the fact that $\p_x u_\Psi+\p_y v_\Psi =0,$  so that
\beno
\bigl(\na\D_k^{\rm h} p_\Psi\ |\ \D_k^{\rm
h}(u_\Psi, v_\Psi)\bigr)_{L^2}=0.\eeno

While due to $\left(u_{\Psi}, v_{\Psi}\right)|_{y=0}=\left(u_\Psi, v_\Psi\right)|_{y=1}=0,$ by applying Poincar\'e inequality, we have
\beq \label{S3eq5a}
\frak{K}\|\D_k(u_\Psi,\e v_\Psi)\|_{L^2}^2\leq \frac12\bigl\|\p_y\D_k(u_\Psi, \e v_\Psi)\bigr\|_{L^2}^2.
\eeq
Then by  using Lemma \ref{lem:Bern} and by multiplying \eqref{S3eq5} by $e^{2\frak{K}t}$ and then  integrating the resulting inequality over $[0,t],$ we achieve
\beq \label{S3eq6}
\begin{split}
&\f12\|\ektp\D_k^{\rm h}(u_\Psi, \e v_\Psi)\|_{L^\infty_t(L^2)}^2+\lam2^k\int_0^t\dot{\eta}(t')\|\ektp\D_k^{\rm
h}(u_\Psi,\e v_\Psi)(t')\|_{L^2}^2\,dt'\\
&+\f12\int_0^te^{2\frak{K}t'}\Bigl(\|\D_k^\h\p_y u_\Psi\|_{L^2}^2+c\e^2 \bigl(2^{2k}\bigl(\|\D_k^\h u_\Psi\|_{L^2}^2+\e^2\|\D_k^\h v_\Psi\|_{L^2}^2\bigr)+\|\D_k^\h \p_y v_\Psi\|_{L^2}^2\bigr)\Bigr)\,dt'\\
&\leq  \bigl\|e^{a|D_x|}\D_k^\h(u_0,\e v_0)\bigr\|_{L^2}^2+\int_0^t\bigl|\bigl(\ektp\D_k^\h\left(u\p_xu\right)_\Psi | \ektp\D_k^\h u_\Psi\bigr)_{L^2}\bigr|\,dt'\\
&\ +\int_0^t\bigl|\bigl(\ektp\D_k^\h\left(v\p_yu\right)_\Psi | \ektp\D_k^\h  u_\Psi\bigr)_{L^2}\bigr|\,dt'
+\e^2\int_0^t\bigl|\bigl(\ektp\D_k^\h\left(u\p_xv\right)_\Psi | \ektp\D_k^\h v_\Psi\bigr)_{L^2}\bigr|\,dt'\\
&\ \qquad\qquad\qquad\qquad\qquad\qquad\qquad\qquad\quad\ \ +\e^2
\int_0^t\bigl|\bigl(\ektp\D_k^\h\left(v\p_yv\right)_\Psi | \ektp\D_k^\h  v_\Psi\bigr)_{L^2}\bigr|\,dt'.
\end{split}
\eeq

In what follows, we shall always assume that $t<T^\ast$ with
$T^\ast$ being determined by \beq\label{eq3.3} T^\ast\eqdefa
\sup\bigl\{\ t>0,\ \ \eta(t) <a/\lam\bigr\}. \eeq So that by
virtue of \eqref{S3eq3}, for any $t<T^\ast,$ there holds the
following convex inequality \beq\label{eq3.4} \Psi(t,\xi)\leq
\Psi(t,\xi-\eta)+\Psi(t,\eta)\quad\mbox{for}\quad \forall\
\xi,\eta\in \R. \eeq

The estimate of \eqref{S3eq6} relies on the following lemmas.

\begin{lem}\label{lem3.1}
{\sl For any $s\in]0,1]$ and $t\leq T^\ast,$ there holds
\beq \label{S3eq7}
\int_0^t\bigl|\bigl(\ektp\D_k^\h(u\p_x w)_\Psi\ |\ \ektp\D_k^\h w_\Psi\bigr)_{L^2}\bigr|\,dt'\lesssim d_k^2
2^{-2ks}\|\ektp w_\Psi\|_{\wt{L}^2_{t,\dot{\eta}(t)}(\cB^{s+\frac12})}^2.
\eeq}

\end{lem}

\begin{lem}\label{lem3.2}
{\sl For any $s\in]0,1]$ and $t\leq T^\ast,$ there holds
\beq \label{S3eq9}
\int_0^t\bigl|\bigl(\ektp\D_k^\h(v\p_y u)_\Psi\ |\ \ektp\D_k^\h u_\Psi\bigr)_{L^2}\bigr|\,dt'\lesssim d_k^2
2^{-2ks}\|\ektp u_\Psi\|_{\wt{L}^2_{t,\dot{\eta}(t)}(\cB^{s+\frac12})}^2.
\eeq}
\end{lem}

\begin{lem}\label{lem3.3}
{\sl For $t\leq T^\ast,$ there holds
\beq \label{S3eq10}
\e^2\int_0^t\bigl|\bigl(\ektp\D_k^\h(v\p_y v)_\Psi\ |\ \ektp\D_k^\h v_\Psi\bigr)_{L^2}\bigr|\,dt'\lesssim d_k^2
2^{-k}\bigl\|\ektp (u_\Psi, \e v_\Psi)\bigr\|_{\wt{L}^2_{t,\dot{\eta}(t)}(\cB^{1})}^2.
\eeq}
\end{lem}

Let us admit the above lemmas for the time being and continue our proof. Indeed, thanks to Lemmas \ref{lem3.1}-\ref{lem3.3},
we deduce from \eqref{S3eq6} that
\beno
\begin{split}
\f12\|&\ektp\D_k^{\rm h}(u_\Psi, \e v_\Psi)\|_{L^\infty_t(L^2)}^2+\lam2^k\int_0^t\dot{\eta}(t')\|\ektp\D_k^{\rm
h}(u_\Psi,\e v_\Psi)(t')\|_{L^2}^2\,dt'\\
&+\f{c}2\int_0^te^{2\frak{K}t'}\Bigl(\|\D_k^\h\p_y(u_\Psi,\e v_\Psi)\|_{L^2}^2+\e^22^{2k}\|\D_k^\h (u_\Psi, \e v_\Psi)\|_{L^2}^2\Bigr)\,dt'\\
\leq & \bigl\|e^{a |D_x|}\D_k^\h(u_0,\e v_0)\bigr\|_{L^2}^2+Cd_k^22^{-k}\bigl\|\ektp(u_\Psi, \e v_\Psi)\bigr\|_{\wt{L}^2_{t,\dot{\eta}(t)}(\cB^{1})}^2.
\end{split}
\eeno
By multiplying the above inequality by $2^k$ and then
taking square root of the resulting inequality,  and  finally by summing up the resulting ones over $\Z,$ we find that for $t\leq T^\ast$
\beno
\begin{split}
\|\ektp&(u_\Psi,\e v_\Psi)\|_{\wt{L}^\infty_t(\cB^{\f12})}+\sqrt{\lam}\|\ektp(u_\Psi,\e v_\Psi)\|_{\wt{L}^2_{t,\dot\eta(t)}(\cB^{1})}
+c\|\ektp\p_y (u_\Psi,\e v_\Psi)\|_{\wt{L}^2_t(\cB^{\f12})}\\
&+c\e^2\bigl\|\ektp(u_\Psi,\e v_\Psi)\|_{\wt{L}^2_t(\cB^{\f32})} \leq \bigl\|e^{a |D_x|}(u_0,\e v_0)\bigr\|_{\cB^{\f12}}
+C\|\ektp(u_\Psi,\e v_\Psi)\|_{\wt{L}^2_{t,\dot\eta(t)}(\cB^{1})}.
\end{split}
\eeno
Taking $\lam =C^2$ in the above inequality leads to
\beq \label{S3eq18a}
\begin{split}
\|\ektp(u_\Psi,\e v_\Psi)&\|_{\wt{L}^\infty_t(\cB^{\f12})}
+c\|\ektp\p_y (u_\Psi,\e v_\Psi)\|_{\wt{L}^2_t(\cB^{\f12})} \\
&+c\e^2\bigl\|\ektp(u_\Psi,\e v_\Psi)\|_{\wt{L}^2_t(\cB^{\f32})}\leq \bigl\|e^{a |D_x|}(u_0,\e v_0)\bigr\|_{\cB^{\f12}}\quad\mbox{for}\ t\leq T^\ast.
\end{split}
\eeq
Then for $t\leq T^\ast,$ we deduce from \eqref{S3eq2} that
\beno
\begin{split}
\eta(t)=&\int_0^t\bigl(\e\|\p_x \ue_\Psi(t')\|_{\cB^{\f12}}+\|\p_y\ue_\Psi(t')\|_{\cB^{\f12}}\bigr)\,dt'\\
\leq &\Bigl(\int_0^te^{-2\frak{K}t'}\,dt'\Bigr)^{\f12}\Bigl(\int_0^t\bigl(\e\|\ektp\p_x \ue_\Psi(t')\|_{\cB^{\f12}}+\|\ektp\p_y\ue_\Psi(t')\|_{\cB^{\f12}}\bigr)^2\,dt'\Bigr)^{\f12}\\
\leq &C\bigl\|\ektp(\e\p_x u_\Psi^\e, \p_yu_\Psi^\e)\bigr\|_{\wt{L}^2_t(\cB^{\frac12})}\\
\leq &C\bigl\|e^{a|D_x|}(u_0,\e v_0)\bigr\|_{\cB^{\f12}}.
\end{split}
\eeno
In particular, if we take $c_0$ in \eqref{S1eq8} to be so small that
\beq\label{S3eq19}
C\bigl\|e^{a|D_x|}(u_0,\e v_0)\bigr\|_{\cB^{\f12}}\leq \f{a}{2\lam},
\eeq
we deduce by a continuous argument that $T^\ast$ determined by \eqref{eq3.3} equals $+\infty$ and \eqref{S3eq18} holds. This completes the proof of Theorem \ref{thm1.1}.
\end{proof}

Now let us present the proof of Lemmas \ref{lem3.1} to \ref{lem3.3}. Indeed, we observe that it amounts to prove these lemmas for $\frak{K}=0.$ Without loss of generality, we may assume that $\widehat{u}\ge 0$ and $\widehat{v}\ge0$ (and similar assumption for the proof of the product law in the rest of this paper, one may check \cite{CGP} for detail).

\begin{proof}[Proof of Lemma \ref{lem3.1}] We first get, by applying Bony's decomposition \eqref{Bony}   for the horizontal variable to $u\p_xw$, that
\beno
u\p_xw=T^\h_{u}\p_xw+T^\h_{\p_xw}u+R^h(u,\p_xw).
\eeno
Accordingly, we shall handle the following three terms:\smallskip

\no $\bullet$ \underline{Estimate of
$\int_0^t\bigl(\D_k^{\rm h}(T^\h_{u}\p_xw)_\Psi\ |\ \D_k^{\rm
h}w_\Psi\bigr)_{L^2}\,dt'$}

Considering the support properties to the Fourier transform of the terms in $T^\h_{u}\p_xw,$ we infer
\beno
\begin{split}
\int_0^t\bigl|\bigl(\D_k^{\rm h}(T^\h_{u}\p_xw)_\Psi\ |&\ \D_k^{\rm
h}w_\Psi\bigr)_{L^2}\bigr|\,dt'\\
\lesssim & \sum_{|k'-k|\leq 4}\int_0^t\|S_{k'-1}^\h u_\Psi(t')\|_{L^\infty}
\|\D_{k'}^\h\p_xw_\Psi(t')\|_{L^2}\|\D_k^\h w_\Psi(t')\|_{L^2}\,dt'.
\end{split}
\eeno
However, it follows  from Lemma \ref{lem:Bern} and Poincar\'e inequality that
\beq\label{S3eq8}
\begin{split}
\|\D_k^\h u_\Psi(t)\|_{L^\infty}\lesssim & 2^{\frac{k}2}\|\D_k^\h u_\Psi(t)\|_{L^2_\h (L^\infty_{\rm v})}\\
\lesssim& 2^{\frac{k}2}\|\D_k^\h u_\Psi(t)\|_{L^2}^{\f12}\|\D_k^\h \p_y u_\Psi(t)\|_{L^2}^{\f12}\\
\lesssim & 2^{\frac{k}2} \|\D_k^\h \p_y u_\Psi(t)\|_{L^2}\lesssim d_j(t)\|\p_y u_\Psi(t)\|_{\cB^{\f12}},
\end{split}
\eeq
so that
\beno
\|S_{k'-1}^\h u_\Psi(t)\|_{L^\infty}\lesssim \|\p_yu_\Psi(t)\|_{\cB^{\f12}},
\eeno
which implies that
\beno
\begin{split}
\int_0^t\bigl|\bigl(\D_k^{\rm h}(T^\h_{u}\p_xw)_\Psi\ |&\ \D_k^{\rm
h}w_\Psi\bigr)_{L^2}\bigr|\,dt'\\
\lesssim & \sum_{|k'-k|\leq 4}2^{k'}\int_0^t\|\p_yu_\Psi(t)\|_{\cB^{\f12}}
\|\D_{k'}^\h w_\Psi(t)\|_{L^2}\|\D_k^\h w_\Psi(t')\|_{L^2}\,dt'.
\end{split}
\eeno
Applying H\"older inequality and using Definition \ref{def1.1} gives
\beno
\begin{split}
\int_0^t\bigl|\bigl(\D_k^{\rm h}(T^\h_{u}\p_xw)_\Psi\ |\ \D_k^{\rm
h}w_\Psi\bigr)_{L^2}\bigr|\,dt'
\lesssim & \sum_{|k'-k|\leq 4}2^{k'}\Bigl(\int_0^t\|\p_yu_\Psi(t')\|_{\cB^{\f12}}\|\D_{k'}^\h w_\Psi(t')\|_{L^2}^2\,dt'\Bigr)^{\f12}\\
&\qquad\qquad\times\Bigl(\int_0^t\|\p_yu_\Psi(t')\|_{\cB^{\f12}} \|\D_k^\h w_\Psi(t')\|_{L^2}^2\,dt'\Bigr)^{\f12}\\
\lesssim & d_k2^{-2ks}\|w_\Psi\|_{\wt{L}^2_{t,\dot{\eta}(t)}(\cB^{s+\frac12})}^2\Bigl(\sum_{|k'-k|\leq 4}d_{k'}2^{(k-k')\left(s-\f12\right)}\Bigr)\\
\lesssim & d_k^22^{-2ks}\|w_\Psi\|_{\wt{L}^2_{t,\dot{\eta}(t)}(\cB^{s+\frac12})}^2.
\end{split}
\eeno

\no $\bullet$ \underline{Estimate of
$\int_0^t\bigl(\D_k^{\rm h}(T^\h_{\p_xw}u)_\Psi\ |\ \D_k^{\rm
h}w_\Psi\bigr)_{L^2}\,dt'$}

Again considering the support properties to the Fourier transform of the terms in $T^\h_{\p_xw}u$ and thanks to  \eqref{S3eq8},
we have
\begin{align*}
\int_0^t\bigl|\bigl(&\D_k^{\rm h}(T^\h_{\p_xw}u)_\Psi\ |\ \D_k^{\rm
h}w_\Psi\bigr)_{L^2}\bigr|\,dt'\\
\lesssim & \sum_{|k'-k|\leq 4}\int_0^t\|S_{k'-1}^\h \p_xw_\Psi(t')\|_{L^\infty_\h(L^2_{\rm v})}
\|\D_{k'}^\h u_\Psi(t')\|_{L^2_\h(L^\infty_{\rm v})}\|\D_k^\h w_\Psi(t')\|_{L^2}\,dt'\\
\lesssim & \sum_{|k'-k|\leq 4}2^{-\frac{k'}2}\int_0^td_{k'}(t)\|S_{k'-1}^\h\p_xw_\Psi(t')\|_{L^\infty_\h(L^2_{\rm v})}\|\p_yu_\Psi(t')\|_{\cB^{\f12}}
\|\D_k^\h w_\Psi(t')\|_{L^2}\,dt'\\
\lesssim & \sum_{|k'-k|\leq 4}d_{k'}2^{-\frac{k'}2}\Bigl(\int_0^t\|S_{k'-1}^\h\p_xw_\Psi(t')\|_{L^\infty_\h(L^2_{\rm v})}^2\|\p_yu_\Psi(t')\|_{\cB^{\f12}}\,dt'\Bigr)^{\frac12}\\
&\qquad\qquad\qquad\qquad\qquad\times \Bigl(\int_0^t\|\D_k^\h w_\Psi(t')\|_{L^2}^2\|\p_yu_\Psi(t')\|_{\cB^{\f12}}\,dt'\Bigr)^{\frac12}.
\end{align*}
Yet we observe from Definition \ref{def1.1} and $s\leq 1$ that
\beno
\begin{split}
\Bigl(\int_0^t&\|S_{k'-1}^\h\p_xw_\Psi(t')\|_{L^\infty_\h(L^2_{\rm v})}^2\|\p_yu_\Psi(t')\|_{\cB^{\f12}}\,dt'\Bigr)^{\frac12}\\
\lesssim
&\sum_{\ell\leq k'-2}2^{\frac{3\ell}2}\Bigl(\int_0^t\|\D_\ell^\h w_\Psi(t')\|_{L^2}^2\|\p_yu_\Psi(t')\|_{\cB^{\f12}}\,dt'\Bigr)^{\frac12}\\
\lesssim
&\sum_{\ell\leq k'-2}d_\ell 2^{\ell(1-s)}\|w_\Psi\|_{\wt{L}^2_{t,\dot{\eta}(t)}(\cB^{s+\frac12})}\\
\lesssim &2^{k'(1-s)}\|w_\Psi\|_{\wt{L}^2_{t,\dot{\eta}(t)}(\cB^{s+\frac12})}.
\end{split}
\eeno
So that it comes out
\beno
\begin{split}
\int_0^t\bigl|\bigl(\D_k^{\rm h}(T^\h_{\p_xw}u)_\Psi\ |\ \D_k^{\rm
h}w_\Psi\bigr)_{L^2}\bigr|\,dt' \lesssim & d_k^22^{-2ks}\|w_\Psi\|_{\wt{L}^2_{t,\dot{\eta}(t)}(\cB^{s+\frac12})}^2.
\end{split}
\eeno

\no $\bullet$ \underline{Estimate of
$\int_0^t\bigl(\D_k^{\rm h}(R^\h(u,\p_xw))_\Psi\ |\ \D_k^{\rm
h}w_\Psi\bigr)_{L^2}\,dt'$}\vspace{0.2cm}

Again considering the support properties to the Fourier transform of the terms in $R^\h(u,\p_xw),$ we get, by applying lemma
\ref{lem:Bern} and \eqref{S3eq8}, that
\beno
\begin{split}
\int_0^t\bigl|&\bigl(\D_k^{\rm h}(R^\h(u,\p_xw))_\Psi\ |\ \D_k^{\rm
h}w_\Psi\bigr)_{L^2}\bigr|\,dt'\\
\lesssim &2^{\f{k}2}\sum_{k'\geq k-3}\int_0^t\|\wt{\D}_{k'}^\h u_\Psi(t')\|_{L^2_\h(L^\infty_{\rm v})}\|{\D}_{k'}^\h \p_x w_\Psi(t')\|_{L^2}\|\D_k^\h w_\Psi(t')\|_{L^2}\,dt'\\
\lesssim & 2^{\f{k}2}\sum_{k'\geq k-3}2^{\f{k'}2}\int_0^t\|\p_yu_\Psi(t')\|_{\cB^{\f12}}\|{\D}_{k'}^\h  w_\Psi(t')\|_{L^2}\|\D_k^\h w_\Psi(t')\|_{L^2}\,dt'.
\end{split}
\eeno
Applying H\"older inequality and using Definition \ref{def1.1} yields
\beno
\begin{split}
\int_0^t\bigl|&\bigl(\D_k^{\rm h}(R^\h(u,\p_xw))_\Psi\ |\ \D_k^{\rm
h}w_\Psi\bigr)_{L^2}\bigr|\,dt'\\
\lesssim & 2^{\f{k}2}\sum_{k'\geq k-3}2^{\f{k'}2}\Bigl(\int_0^t\|{\D}_{k'}^\h  w_\Psi(t')\|_{L^2}^2\|\p_yu_\Psi(t')\|_{\cB^{\f12}}\,dt'\Bigr)^{\f12}
\\
&\qquad\qquad\qquad\qquad\times \Bigl(\int_0^t\|\D_k^\h w_\Psi(t')\|_{L^2}^2\|\p_yu_\Psi(t')\|_{\cB^{\f12}}\,dt'\Bigr)^{\f12}\\
\lesssim &d_k2^{-2{k}s} \|w_\Psi\|_{\wt{L}^2_{t,\dot{\eta}(t)}(\cB^{s+\frac12})}^2\Bigl(\sum_{k'\geq k-3}d_{k'}2^{(k-k')s}\Bigr)\\
\lesssim &d_k^22^{-2ks}\|w_\Psi\|_{\wt{L}^2_{t,\dot{\eta}(t)}(\cB^{s+\frac12})}^2,
\end{split}
\eeno
where we used the fact that $s>0$ in the last step.

By summing up the above estimates, we conclude the proof of \eqref{S3eq7}.
\end{proof}

\begin{rmk}\label{rmk3.1}
In the  particular case when $w=u$ in \eqref{S3eq7}, \eqref{S3eq7} holds for any $s>0,$ that is
\beq \label{S3eq7a}
\int_0^t\bigl|\bigl(\ektp\D_k^\h(u\p_x u)_\Psi\ |\ \ektp\D_k^\h u_\Psi\bigr)_{L^2}\bigr|\,dt'\lesssim d_k^2
2^{-2ks}\|\ektp u_\Psi\|_{\wt{L}^2_{t,\dot{\eta}(t)}(\cB^{s+\frac12})}^2.
\eeq
It follows from the proof of Lemma \ref{lem3.1} that we only need to prove
\beq \label{S3eq7b}
\int_0^t\bigl|\bigl(\D_k^{\rm h}(T^\h_{\p_xu}u)_\Psi\ |\ \D_k^{\rm
h}u_\Psi\bigr)_{L^2}\bigr|\,dt' \lesssim d_k^22^{-2ks}\|u_\Psi\|_{\wt{L}^2_{t,\dot{\eta}(t)}(\cB^{s+\frac12})}^2\quad\mbox{for any}\ s>0.
\eeq
Indeed in view of \eqref{S3eq8},
we infer
\beno
\begin{split}
\int_0^t\bigl|\bigl(&\D_k^{\rm h}(T^\h_{\p_xu}u)_\Psi\ |\ \D_k^{\rm
h}u_\Psi\bigr)_{L^2}\bigr|\,dt'\\
\lesssim & \sum_{|k'-k|\leq 4}\int_0^t\|S_{k'-1}^\h \p_xu_\Psi(t')\|_{L^\infty}
\|\D_{k'}^\h u_\Psi(t')\|_{L^2}\|\D_k^\h u_\Psi(t')\|_{L^2}\,dt'\\
\lesssim & \sum_{|k'-k|\leq 4}2^{{k'}}\int_0^t\|\p_yu_\Psi(t')\|_{\cB^{\f12}}
\|\D_{k'}^\h u_\Psi(t')\|_{L^2}
\|\D_k^\h u_\Psi(t')\|_{L^2}\,dt'\\
\lesssim & \sum_{|k'-k|\leq 4}2^{{k'}}\Bigl(\int_0^t\|\D_{k'}^\h u_\Psi(t')\|_{L^2}^2\|\p_yu_\Psi(t')\|_{\cB^{\f12}}\,dt'\Bigr)^{\frac12}\\
&\qquad\qquad\qquad\qquad\qquad\times \Bigl(\int_0^t\|\D_k^\h u_\Psi(t')\|_{L^2}^2\|\p_yu_\Psi(t')\|_{\cB^{\f12}}\,dt'\Bigr)^{\frac12},
\end{split}
\eeno
which leads to \eqref{S3eq7b}.
\end{rmk}

\begin{proof}[Proof of Lemma \ref{lem3.2}] We first get, by applying Bony's decomposition \eqref{Bony}  for the horizontal variable to  $v\p_yu$, that
\beno
v\p_yu=T^\h_{v}\p_yu+T^\h_{\p_yu}v+R^h(v,\p_yu).
\eeno
Accordingly, we shall handle the following three terms:\smallskip

\no $\bullet$ \underline{Estimate of
$\int_0^t\bigl(\D_k^{\rm h}(T^\h_{v}\p_yu)_\Psi\ |\ \D_k^{\rm
h}u_\Psi\bigr)_{L^2}\,dt'$}

We first observe that
\beno
\begin{split}
\int_0^t\bigl|\bigl(\D_k^{\rm h}(T^\h_{v}\p_yu)_\Psi\ |&\ \D_k^{\rm
h}u_\Psi\bigr)_{L^2}\bigr|\,dt'\\
\lesssim & \sum_{|k'-k|\leq 4}\int_0^t\|S_{k'-1}^\h v_\Psi(t')\|_{L^\infty}
\|\D_{k'}^\h\p_yu_\Psi(t)\|_{L^2}\|\D_k^\h u_\Psi(t')\|_{L^2}\,dt'\\
\lesssim & \sum_{|k'-k|\leq 4} d_{k'}2^{-\f{k'}2}\int_0^t\|S_{k'-1}^\h v_\Psi(t')\|_{L^\infty}\|\p_yu_\Psi(t')\|_{\cB^{\f12}}
\|\D_k^\h u_\Psi(t')\|_{L^2}\,dt'.
\end{split}
\eeno
Due to $\p_xu+\p_yv=0$ and \eqref{S1eq5},  we write $v(t,x,y)=-\int_0^y\p_x u(t,x,y')\,dy'.$ Then we deduce
from Lemma \ref{lem:Bern} that
\beq \label{S3eq11}
\begin{split}
\|\D_k^\h v_\Psi(t)\|_{L^\infty}\leq & \int_0^1\|\D_k^\h\p_xu_\Psi(t,\cdot,y')\|_{L^\infty_\h}\,dy'\\
\lesssim &2^{\f{3k}2}\int_0^1\|\D_k^\h u_\Psi(t,\cdot,y')\|_{L^2_h}\,dy'\lesssim 2^{\f{3k}2}\|\D_k^\h u_\Psi(t)\|_{L^2},
\end{split}
\eeq
from which and $s\leq 1,$ we infer
\beq \label{S3eq12}
\begin{split}
\Bigl(\int_0^t&\|S_{k'-1}^\h v_\Psi(t')\|_{L^\infty}^2\|\p_yu_\Psi(t')\|_{\cB^{\f12}}\,dt'\Bigr)^{\f12}\\
\leq &\sum_{\ell\leq k'-2}2^{\f{3\ell}2}\Bigl(\int_0^t\|\D_\ell^\h u_\Psi(t)\|_{L^2}^2\|\p_yu_\Psi(t')\|_{\cB^{\f12}}\,dt'\Bigr)^{\f12}\\
\lesssim &\sum_{\ell\leq k'-2}d_\ell 2^{\ell(1-s)}\|u_\Psi\|_{\wt{L}^2_{t,\dot{\eta}(t)}(\cB^{s+\f12})}\\
\lesssim &2^{{k'}(1-s)}\|u_\Psi\|_{\wt{L}^2_{t,\dot{\eta}(t)}(\cB^{s+\frac12})}.
\end{split}
\eeq
Consequently, by virtue of Definition \ref{def1.1}, we obtain
\beno
\begin{split}
\int_0^t\bigl|\bigl(\D_k^{\rm h}(T^\h_{v}\p_yu)_\Psi\ |&\ \D_k^{\rm
h}u_\Psi\bigr)_{L^2}\bigr|\,dt'\\
\lesssim & \sum_{|k'-k|\leq 4}d_{k'}2^{-\f{k'}2}\Bigl(\int_0^t\|S_{k'-1}^\h v_\Psi(t')\|_{L^\infty}^2\|\p_yu_\Psi(t')\|_{\cB^{\f12}}\,dt'\Bigr)^{\f12}\\
&\qquad\qquad\qquad\quad\times\Bigl(\int_0^t\|\D_k^\h u_\Psi(t')\|_{L^2}^2\|\p_yu_\Psi(t')\|_{\cB^{\f12}}\,dt'\Bigr)^{\f12}\\
\lesssim &d_k^22^{-2ks}\|u_\Psi\|_{\wt{L}^2_{t,\dot{\eta}(t)}(\cB^{s+\f12})}^2.
\end{split}
\eeno

\no $\bullet$ \underline{Estimate of
$\int_0^t\bigl(\D_k^{\rm h}(T^\h_{\p_yu}v)_\Psi\ |\ \D_k^{\rm
h}u_\Psi\bigr)_{L^2}\,dt'$}

Notice that
\beno
\begin{split}
\int_0^t\bigl|\bigl(&\D_k^{\rm h}(T^\h_{\p_yu}v)_\Psi\ |\ \D_k^{\rm h}u_\Psi\bigr)_{L^2}\bigr|\,dt'\\
\lesssim & \sum_{|k'-k|\leq 4}\int_0^t\|S_{k'-1}^\h \p_y u_\Psi(t')\|_{L^\infty_\h(L^2_{\rm v})}
\|\D_{k'}^\h v_\Psi(t)\|_{L^2_\h(L^\infty_{\rm v})}\|\D_k^\h u_\Psi(t')\|_{L^2}\,dt',
\end{split}
\eeno
which together with \eqref{S3eq11} ensures that
\beno
\begin{split}
\int_0^t\bigl|\bigl(&\D_k^{\rm h}(T^\h_{\p_yu}v)_\Psi\ |\ \D_k^{\rm
h}u_\Psi\bigr)_{L^2}\bigr|\,dt'\\
\lesssim & \sum_{|k'-k|\leq 4}2^{k'}\int_0^t\|\p_yu_\Psi(t')\|_{\cB^{\f12}}\|\D_{k'}^\h u_\Psi(t')\|_{L^2}
\|\D_k^\h u_\Psi(t')\|_{L^2}\,dt'\\
\lesssim & \sum_{|k'-k|\leq 4}2^{k'}\Bigl(\int_0^t\|\D^\h_{k'}u_\Psi(t')\|_{L^2}^2\|\p_yu_\Psi(t')\|_{\cB^{\f12}}\,dt'\Bigr)^{\frac12}\\
&\qquad\qquad\qquad\times
\Bigl(\int_0^t\|\D_k^\h u_\Psi(t')\|_{L^2}^2\|\p_yu_\Psi(t')\|_{\cB^{\f12}}\,dt'\Bigr)^{\frac12}.
\end{split}
\eeno
Then thanks to Definition \ref{def1.1}, we arrive at
\beno
\begin{split}
\int_0^t\bigl|\bigl(\D_k^{\rm h}(T^\h_{\p_yu}v)_\Psi\ |\ \D_k^{\rm
h}u_\Psi\bigr)_{L^2}\bigr|\,dt' \lesssim & d_k^22^{-2ks}\|u_\Psi\|_{\wt{L}^2_{t,\dot{\eta}(t)}(\cB^{s+\frac12})}^2.
\end{split}
\eeno

\no $\bullet$ \underline{Estimate of
$\int_0^t\bigl(\D_k^{\rm h}(R^\h(v,\p_yu))_\Psi\ |\ \D_k^{\rm
h}u_\Psi\bigr)_{L^2}\,dt'$}\vspace{0.2cm}

We get, by applying lemma
\ref{lem:Bern} and \eqref{S3eq11}, that
\beno
\begin{split}
\int_0^t\bigl|&\bigl(\D_k^{\rm h}(R^\h(v,\p_yu))_\Psi\ |\ \D_k^{\rm
h}u_\Psi\bigr)_{L^2}\bigr|\,dt'\\
\lesssim &2^{\f{k}2}\sum_{k'\geq k-3}\int_0^t\|{\D}_{k'}^\h v_\Psi(t')\|_{L^2_\h(L^\infty_{\rm v})}\|\wt{\D}_{k'}^\h \p_y u_\Psi(t')\|_{L^2}\|\D_k^\h u_\Psi(t')\|_{L^2}\,dt'\\
\lesssim & 2^{\f{k}2}\sum_{k'\geq k-3}2^{\f{k'}2}\int_0^t\|{\D}_{k'}^\h u_\Psi(t')\|_{L^2}\|\p_yu_\Psi(t')\|_{\cB^{\f12}}\|\D_k^\h u_\Psi(t')\|_{L^2}\,dt'\\
\lesssim & 2^{\f{k}2}\sum_{k'\geq k-3}2^{\f{k'}2}\Bigl(\int_0^t\|{\D}_{k'}^\h  u_\Psi(t')\|_{L^2}^2\|\p_yu_\Psi(t')\|_{\cB^{\f12}}\,dt'\Bigr)^{\f12}
\\
&\qquad\qquad\qquad\times \Bigl(\int_0^t\|\D_k^\h u_\Psi(t')\|_{L^2}^2\|\p_yu_\Psi(t')\|_{\cB^{\f12}}\,dt'\Bigr)^{\f12},
\end{split}
\eeno
which together with Definition \ref{def1.1} and $s>0$ ensures that
\beno
\begin{split}
\int_0^t\bigl|\bigl(\D_k^{\rm h}(R^\h(v,\p_yu)_\Psi\ |\ \D_k^{\rm
h}u_\Psi\bigr)_{L^2}\bigr|\,dt'
\lesssim &d_k2^{-{2ks}} \|u_\Psi\|_{\wt{L}^2_{t,\dot{\eta}(t)}(\cB^{s+\frac12})}^2\Bigl(\sum_{k'\geq k-3}d_{k'}2^{(k-k')s}\Bigr)\\
\lesssim &d_k^22^{-2ks}\|u_\Psi\|_{\wt{L}^2_{t,\dot{\eta}(t)}(\cB^{s+\frac12})}^2.
\end{split}
\eeno

By summing up the above estimates, we  achieve \eqref{S3eq9}.
\end{proof}

\begin{proof}[Proof of Lemma \ref{lem3.3}] We first get, by applying Bony's decomposition \eqref{Bony}  for the horizontal variable to $v\p_yv$, that
\beno
v\p_yv=T^\h_{v}\p_yv+T^\h_{\p_yv}v+R^h(v,\p_yv).
\eeno
Let us handle the following three terms:\smallskip

\no $\bullet$ \underline{Estimate of
$\int_0^t\bigl(\D_k^{\rm h}(T^\h_{v}\p_yv)_\Psi\ |\ \D_k^{\rm
h}v_\Psi\bigr)_{L^2}\,dt'$}

Due to $\p_yv=-\p_xu,$ one has
\beno
\begin{split}
\e^2\int_0^t\bigl|\bigl(&\D_k^{\rm h}(T^\h_{v}\p_yv)_\Psi\ |\ \D_k^{\rm
h}v_\Psi\bigr)_{L^2}\bigr|\,dt'\\
\lesssim & \e^2\sum_{|k'-k|\leq 4}\int_0^t\|S_{k'-1}^\h v_\Psi(t')\|_{L^\infty}
\|\D_{k'}^\h\p_yv_\Psi(t)\|_{L^2}\|\D_k^\h v_\Psi(t')\|_{L^2}\,dt'\\
\lesssim & \e\sum_{|k'-k|\leq 4}2^{-\f{k'}2}\int_0^t\|S_{k'-1}^\h v_\Psi(t')\|_{L^\infty}\e\|\p_xu_\Psi(t')\|_{\cB^{\f12}}
\|\D_k^\h v_\Psi(t')\|_{L^2}\,dt'\\
\lesssim & \e\sum_{|k'-k|\leq 4}2^{-\f{k'}2}\Bigl(\int_0^t\|S_{k'-1}^\h v_\Psi(t')\|_{L^\infty}^2\e\|\p_xu_\Psi(t')\|_{\cB^{\f12}}\,dt'\Bigr)^{\f12}\\
&\qquad\qquad\qquad\qquad\times \Bigl(\int_0^t\|\D_k^\h v_\Psi(t')\|_{L^2}^2\e\|\p_xu_\Psi(t')\|_{\cB^{\f12}}\,dt'\Bigr)^{\f12}.
\end{split}
\eeno
Yet we get, by a similar derivation of \eqref{S3eq12}, that
\beno
\begin{split}
\Bigl(\int_0^t&\|S_{k'-1}^\h v_\Psi(t')\|_{L^\infty}^2\e\|\p_x u_\Psi(t')\|_{\cB^{\f12}}\,dt'\Bigr)^{\f12}
\lesssim d_{k'}2^{\f{k'}2}\|u_\Psi\|_{\wt{L}^2_{t,\dot{\eta}(t)}(\cB^{1})}.
\end{split}
\eeno
Hence we deduce from Definition \ref{def1.1} that
\beno
\e^2\int_0^t\bigl|\bigl(\D_k^{\rm h}(T^\h_{v}\p_yv)_\Psi\ |\ \D_k^{\rm
h}v_\Psi\bigr)_{L^2}\bigr|\,dt'\leq d_k^22^{-k}\|u_\Psi\|_{\wt{L}^2_{t,\dot{\eta}(t)}(\cB^{1})}\e\|v_\Psi\|_{\wt{L}^2_{t,\dot{\eta}(t)}(\cB^{1})}.
\eeno

\no $\bullet$ \underline{Estimate of
$\int_0^t\bigl(\D_k^{\rm h}(T^\h_{\p_yv}v)_\Psi\ |\ \D_k^{\rm
h}v_\Psi\bigr)_{L^2}\,dt'$}

Notice that
\beno
\begin{split}
\int_0^t\bigl|\bigl(&\D_k^{\rm h}(T^\h_{\p_yv}v)_\Psi\ |\ \D_k^{\rm h}v_\Psi\bigr)_{L^2}\bigr|\,dt'\\
\lesssim & \sum_{|k'-k|\leq 4}\int_0^t\|S_{k'-1}^\h \p_x u_\Psi(t')\|_{L^\infty}
\|\D_{k'}^\h v_\Psi(t)\|_{L^2}\|\D_k^\h v_\Psi(t')\|_{L^2}\,dt',
\end{split}
\eeno
which together with \eqref{S3eq8} ensures that
\beno
\begin{split}
\int_0^t\bigl|\bigl(&\D_k^{\rm h}(T^\h_{\p_yv}v)_\Psi\ |\ \D_k^{\rm h}v_\Psi\bigr)_{L^2}\bigr|\,dt'\\
\lesssim & \sum_{|k'-k|\leq 4}2^{k'}\int_0^t\|\p_y u_\Psi(t')\|_{\cB^{\f12}}
\|\D_{k'}^\h v_\Psi(t)\|_{L^2}\|\D_k^\h v_\Psi(t')\|_{L^2}\,dt'\\
\lesssim & \sum_{|k'-k|\leq 4}2^{k'}\Bigl(\int_0^t
\|\D_{k'}^\h v_\Psi(t)\|_{L^2}^2\|\p_y u_\Psi(t')\|_{\cB^{\f12}}\,dt'\Bigr)^{\f12}\\
&\qquad\qquad\qquad\qquad \times \Bigl(\int_0^t\|\D_k^\h v_\Psi(t')\|_{L^2}^2\|\p_y u_\Psi(t')\|_{\cB^{\f12}}\,dt'\Bigr)^{\f12}
\end{split}
\eeno
Then thanks to Definition \ref{def1.1}, we arrive at
\beno
\begin{split}
\int_0^t\bigl|\bigl(\D_k^{\rm h}(T^\h_{\p_yv}v)_\Psi\ |\ \D_k^{\rm
h}v_\Psi\bigr)_{L^2}\bigr|\,dt' \lesssim & d_k^22^{-k}\|v_\Psi\|_{\wt{L}^2_{t,\dot{\eta}(t)}(\cB^{1})}^2.
\end{split}
\eeno

\no $\bullet$ \underline{Estimate of
$\int_0^t\bigl(\D_k^{\rm h}(R^\h(v,\p_yv))_\Psi\ |\ \D_k^{\rm
h}v_\Psi\bigr)_{L^2}\,dt'$}\vspace{0.2cm}

Due to $\p_xu+\p_yv=0,$ we get, by applying lemma
\ref{lem:Bern} and \eqref{S3eq11}, that
\beno
\begin{split}
\int_0^t\bigl|&\bigl(\D_k^{\rm h}(R^\h(v,\p_yv))_\Psi\ |\ \D_k^{\rm
h}v_\Psi\bigr)_{L^2}\bigr|\,dt'\\
\lesssim &2^{\f{k}2}\sum_{k'\geq k-3}\int_0^t\|{\D}_{k'}^\h v_\Psi(t')\|_{L^2}\|\wt{\D}_{k'}^\h \p_x u_\Psi(t')\|_{L^2_\h(L^\infty_{\rm v})}\|\D_k^\h v_\Psi(t')\|_{L^2}\,dt'\\
\lesssim & 2^{\f{k}2}\sum_{k'\geq k-3}2^{\f{k'}2}\int_0^t\|{\D}_{k'}^\h v_\Psi(t')\|_{L^2}\|\p_yu_\Psi(t')\|_{\cB^{\f12}}\|\D_k^\h v_\Psi(t')\|_{L^2}\,dt'\\
\lesssim & 2^{\f{k}2}\sum_{k'\geq k-3}2^{\f{k'}2}\Bigl(\int_0^t\|{\D}_{k'}^\h  v_\Psi(t')\|_{L^2}^2\|\p_yu_\Psi(t')\|_{\cB^{\f12}}\,dt'\Bigr)^{\f12}
\\
&\qquad\qquad\qquad\times \Bigl(\int_0^t\|\D_k^\h v_\Psi(t')\|_{L^2}^2\|\p_yu_\Psi(t')\|_{\cB^{\f12}}\,dt'\Bigr)^{\f12},
\end{split}
\eeno
which together with Definition \ref{def1.1} and $s>0$ ensures that
\beno
\begin{split}
\int_0^t\bigl|\bigl(\D_k^{\rm h}(R^\h(v,\p_yu))_\Psi\ |\ \D_k^{\rm
h}u_\Psi\bigr)_{L^2}\bigr|\,dt'
\lesssim d_k^22^{-k}\|v_\Psi\|_{\wt{L}^2_{t,\dot{\eta}(t)}(\cB^{1})}^2.
\end{split}
\eeno

By summing up the above estimates, we obtain \eqref{S3eq10}. This concludes the proof of Lemma \ref{lem3.3}.
\end{proof}

\renewcommand{\theequation}{\thesection.\arabic{equation}}
\setcounter{equation}{0}
\section{Global well-posedness of the system \eqref{S1eq6}}\label{sect4}

In this section, we study the global well-posedness of the hydrostatic approximate equations \eqref{S1eq6} with small analytic data.

Due to the compatibility condition
$\pa_x\int_0^1u_0dy=0$, we deduce from $\pa_xu+\pa_yv=0$ that
\ben
\pa_x\int_0^1u(t,x,y)dy=0
\een
so that by integrating the equation $\pa_tu+u\pa_xu+v\pa_yu-\pa_y^2u+\pa_xp=0$ for $y\in [0,1]$ and using the fact
that $\pa_y p=0,$ we obtain
\ben
\pa_x^2p=\pa_x\Big(\pa_yu(t,x,1)-\pa_yu(t,x,0)-\pa_x\int_0^1u^2(t,x,y)dy\Big).
\een

We define
\beq \label{S4eq1}
u_\Phi(t,x,y)\eqdefa \cF_{\xi\to
x}^{-1}\bigl(e^{\Phi(t,\xi)}\widehat{u}(t,\xi,y)\bigr) \with \Phi(t,\xi)\eqdefa (a-\lam \tht(t))|\xi|,
\eeq
where the quantity $\theta(t)$ describes the evolution of
the analytic band of $u,$ which is determined by
\begin{equation}\label{def:theta}
 \dot{\tht}(t)=\|\pa_yu_\Phi(t)\|_{\cB^{\f12}}\with \tht|_{t=0}=0.
\end{equation}

\begin{proof}[Proof of Theorem \ref{th1.2}]
In view of \eqref{S1eq6} and \eqref{S4eq1}, we observe
that $u_\Phi$ verifies
\beq\label{S4eq2}
\begin{split}
&\p_tu_\Phi+\lam\dot{\tht}(t)|D_x|u_\Phi+(u\pa_x u)_\Phi+(v\pa_yu)_\Phi-\p_{y}^2u_\Phi+\pa_xp_\Phi=0,
\end{split} \eeq
where $|D_x|$ denotes the Fourier multiplier with symbol $|\xi|.$

By applying $\D_k^\h$ to \eqref{S4eq2} and taking $L^2$ inner product of the resulting equation with $\D_k^\h u_\Phi,$ we find
\beq \label{S4eq5}
\begin{split}
\f12\f{d}{dt}&\|\D_k^{\rm h}u_\Phi(t)\|_{L^2}^2+\lam\dot{\theta}\bigl(|D_x|\D_k^{\rm
h}u_\Phi\ |\ \D_k^{\rm h}u_\Phi\bigr)_{L^2}+\|\D_k^\h\p_y\up\|_{L^2}^2\\
=&-\bigl(\D_k^\h\left(u\p_xu\right)_\Phi | \D_k^\h u_\Phi\bigr)_{L^2}-\bigl(\D_k^\h\left(v\p_yu\right)_\Phi | \D_k^\h u_\Phi\bigr)_{L^2}-\bigl(\D_k^\h \p_x p_\Phi  | \D_k^\h u_\Phi\bigr)_{L^2}.
\end{split}
\eeq
Thanks to \eqref{S1eq7} and $\p_xu+\p_yv=0,$ we get, by using integration by parts, that
\beno
\begin{split}
\big(\D_k^\h\pa_xp_\Phi | \D_k^\h u_\Phi\big)_{L^2}=&-\big(\D_k^\h p_\Phi | \D_k^\h\pa_xu_\Phi\big)_{L^2}\\
=&\big(\D_k^\h p_\Phi | \D_k^\h\pa_yv_\Phi\big)_{L^2}
=-\big(\D_k^\h \pa_yp_\Phi | \D_k^\h v_\Phi\big)_{L^2}=0.
\end{split}
\eeno
Then
by  using Lemma \ref{lem:Bern}, \eqref{S3eq5a} and by multiplying \eqref{S4eq5} by $e^{2\frak{K}t}$ and then  integrating the resulting inequality over $[0,t],$ we achieve
\beq \label{S4eq6}
\begin{split}
\f12\|&\ektp\D_k^{\rm h}u_\Phi\|_{L^\infty_t(L^2)}^2+\lam2^k\int_0^t\dot{\theta}(t')\|\ektp\D_k^{\rm
h}u_\Phi(t')\|_{L^2}^2\,dt'+\f12\|\ektp\D_k^\h\p_y\up\|_{L^2_t(L^2)}^2\\
\leq & \bigl\|e^{a|D_x|}\D_k^\h u_0\bigr\|_{L^2}^2+\int_0^t\bigl|\bigl(\ektp\D_k^\h\left(u\p_xu\right)_\Phi | \ektp\D_k^\h u_\Phi\bigr)_{L^2}\bigr|\,dt'\\
&\qquad\qquad\qquad\quad+\int_0^t\bigl|\bigl(\ektp\D_k^\h\left(v\p_yu\right)_\Phi | \ektp\D_k^\h u_\Phi\bigr)_{L^2}\bigr|\,dt'.
\end{split}
\eeq

In what follows, we shall always assume that $t<T^\star$ with
$T^\star$ being determined by \beq\label{eq4.3} T^\star\eqdefa
\sup\bigl\{\ t>0,\ \ \theta(t) <a/\lam\bigr\}. \eeq So that by
virtue of \eqref{S4eq1}, for any $t\leq T^\star,$ there holds the
following convex inequality \beq\label{eq4.4} \Phi(t,\xi)\leq
\Phi(t,\xi-\eta)+\Phi(t,\eta)\quad\mbox{for}\quad \forall\
\xi,\eta\in \R. \eeq
Then we deduce from Lemma \ref{lem3.1} that for any $s\in ]0,1]$ and $t\leq T^\star$
\beno
\int_0^t\bigl|\bigl(\ektp\D_k^\h(u\p_x u)_\Phi\ |\ \ektp\D_k^\h u_\Phi\bigr)_{L^2}\bigr|\,dt'\lesssim d_k^2
2^{-2ks}\|\ektp u_\Phi\|_{\wt{L}^2_{t,\dot{\theta}(t)}(\cB^{s+\frac12})}^2.
\eeno
Whereas it follows from Lemma \ref{lem3.2} that for any $s\in ]0,1]$ and $t\leq T^\star$
\beno
\int_0^t\bigl|\bigl(\ektp\D_k^\h(v\p_y u)_\Phi\ |\ \ektp\D_k^\h u_\Phi\bigr)_{L^2}\bigr|\,dt'\lesssim d_k^2
2^{-2ks}\|\ektp u_\Phi\|_{\wt{L}^2_{t,\dot{\theta}(t)}(\cB^{s+\frac12})}^2.
\eeno
Inserting the above estimates into \eqref{S4eq6} gives rise to
\beno
\begin{split}
\f12\|\ektp\D_k^{\rm h}u_\Phi\|_{L^\infty_t(L^2)}^2+\lam2^k\int_0^t&\dot{\theta}(t')\|\ektp\D_k^{\rm
h}u_\Phi(t')\|_{L^2}^2\,dt'+\f12\|\ektp\D_k^\h\p_y\up\|_{L^2_t(L^2)}^2\\
&\qquad\quad\leq  \bigl\|e^{a|D_x|}\D_k^\h u_0\bigr\|_{L^2}^2+Cd_k^22^{-2ks}\|\ektp u_\Phi\|_{\wt{L}^2_{t,\dot{\theta}(t)}(\cB^{s+\frac12})}^2.
\end{split}
\eeno
Then for any $s\in ]0,1],$ by multiplying the above inequality by $2^{2ks}$ and then
taking square root of the resulting inequality, and finally by summing up the resulting ones over $\Z,$ we obtain
\beno
\begin{split}
\|\ektp u_\Phi\|_{\wt{L}^\infty_t(\cB^{s})}+\sqrt{\lam}\|\ektp u_\Phi\|_{\wt{L}^2_{t,\dot\theta(t)}(\cB^{s+\f12})}
&+\|\ektp\p_y u_\Phi\|_{\wt{L}^2_t(\cB^{s})}\\
& \leq \bigl\|e^{a|D_x|}u_0\bigr\|_{\cB^{s}}
+C\|\ektp u_\Phi\|_{\wt{L}^2_{t,\dot\theta(t)}(\cB^{s+\f12})}.
\end{split}
\eeno
Taking $\lam =C^2$ in the above inequality leads to
\beq\label{S4eq10}
\|\ektp u_\Phi\|_{\wt{L}^\infty_t(\cB^{s})}
+\|\ektp\p_y u_\Phi\|_{\wt{L}^2_t(\cB^{s})}\leq \bigl\|e^{a|D_x|}u_0\bigr\|_{\cB^{s}}\quad\mbox{for} \ \ s\in ]0,1]\andf t\leq T^\star.
\eeq
In particular, we deduce from \eqref{S4eq10} for $s=\f12$ and \eqref{def:theta} that
\beno
\begin{split}
\theta(t)=&\int_0^t\|\p_y u_\Phi(t')\|_{\cB^{\f12}}\,dt'\\
\leq &\Bigl(\int_0^te^{-2\frak{K}t'}\,dt'\Bigr)^{\f12}\Bigl(\int_0^t\|\ektp\p_y u_\Phi(t')\|_{\cB^{\f12}}^2\,dt'\Bigr)^{\f12}\\
\leq &C\|\ektp\p_y u_\Phi\|_{\wt{L}^2_t(\cB^{\f12})}\leq C\bigl\|e^{a|D_x|}u_0\bigr\|_{\cB^{\f12}}.
\end{split}
\eeno
Then if we take $c_1$ in \eqref{S1eq8a} to be so small that
\beq\label{S3eq19}
C\bigl\|e^{a|D_x|}u_0\bigr\|_{\cB^\f12}\leq \f{a}{2\lam},
\eeq
we deduce by a continuous argument that $T^\star$ determined by \eqref{eq4.3} equals $+\infty$ and \eqref{S1eq12} holds.
Then Theorem \ref{th1.2} is proved provided that we present the proof of \eqref{S1eq13}, which replies on the the following propositions.

\begin{prop}\label{prop4.1}
{\sl Under the assumption of \eqref{S1eq14}, for any $s>0,$ there exists a positive constant $C$ so that for $\lam=C^2\bigl(1+\bigl\|e^{a|D_x|}u_0\bigr\|_{\cB^{\f32}}\bigr),$ there holds
\beq\label{S4eq12}
\|\ekt u_\Phi\|_{\wt{L}^\infty(\R^+;\cB^{s})}+\|\ekt\p_y u_\Phi\|_{\wt{L}^2(\R^+;\cB^{s})} \leq C\bigl\|e^{a|D_x|}u_0\bigr\|_{\cB^{s}}
\eeq}
\end{prop}

\begin{prop}\label{prop4.2}
{\sl Under the assumption of \eqref{S1eq14}, for any $s>0,$ there exists a positive constant $C$ so that for $\lam=C^2\bigl(1+\bigl\|e^{a|D_x|}u_0\bigr\|_{\cB^{\f32}}\bigr),$ there holds 
\beq\label{S4eq13}
\|\ektp \p_y u_\Phi\|_{\wt{L}^\infty_t(\cB^{s})}
+\|\ektp\p_y^2 u_\Phi\|_{\wt{L}^2_t(\cB^{s})}\\
 \ \leq
C\Bigl(\bigl\|e^{a|D_x|}\p_yu_0\bigr\|_{\cB^{s}}
+\bigl\|e^{a|D_x|}u_0\bigr\|_{\cB^{s+1}}\Bigr).
\eeq
}
\end{prop}

We admit the above propositions for the time being and continue our proof of Theorem \ref{th1.2}.

As a matter of fact, it remains to present the estimate of $\|\ekt(\pa_tu)_\Phi\|_{\wt{L}^2(\R^+;\cB^{\f32})}.$  Indeed,
by  applying $\D_k^\h$ to \eqref{S1eq6} and then taking $L^2$ inner product of resulting equation  with  $e^{2\frak{K}t}\D_k^\h(\p_tu)_\Phi,$ we obtain
\beno
\begin{split}
\|\ekt \D_k^\h\left(\p_tu\right)_\Phi\|_{L^2}^2=&e^{2\frak{K}t}\bigl(\D_k^\h\pa_y^2 u_\Phi  | \D_k^\h(\p_tu)_\Phi\bigr)_{L^2}\\
&-e^{2\frak{K}t}\bigl(\D_k^\h(u\pa_x u)_\Phi  | \D_k^\h(\p_tu)_\Phi\bigr)_{L^2}-e^{2\frak{K}t}\bigl(\D_k^\h(v\pa_yu)_\Phi  | \D_k^\h(\p_tu)_\Phi\bigr)_{L^2},\end{split} \eeno
from which, we deduce that
\begin{align*}
\|\ektp \D_k^\h\left(\p_tu\right)_\Phi\|_{L^2_t(L^2)}\leq C\Bigl(&\|\ektp \D_k^\h\p_y^2u_\Phi\|_{L^2_t(L^2)}\\
&+\bigl\|\ektp (u\pa_x u)_\Phi\bigr\|_{L^2_t(L^2)}
+\bigl\|\ektp(v\pa_yu)_\Phi\|_{L^2_t(L^2)}\Bigr).
\end{align*}
This gives rise to
\beq \label{S4eq13a}
\begin{split}
\bigl\|\ektp(\pa_tu)_\Phi\bigr\|_{\wt{L}^2_t(\cB^{\f32})}\leq
C\Bigl(&\bigl\|\ektp\pa_y^2u_\Phi\bigr\|_{\wt{L}^2_t(\cB^{\f32})}\\
&+\bigl\|\ektp(u\p_xu)_\Phi\bigr\|_{\wt{L}^2_t(\cB^{\f32})}
+\bigl\|\ektp(v\p_yu)_\Phi\bigr\|_{\wt{L}^2_t(\cB^{\f32})}\Bigr).
\end{split}
\eeq
Yet it follows from the law of product in anisotropic Besov space and Poincare inequality that
\begin{align*}
\bigl\|\ektp(u\p_xu)_\Phi\bigr\|_{\wt{L}^2_t(\cB^{\f32})}\lesssim &\|u_\Phi\|_{\wt{L}^\infty_t(\cB^{\f12})}\bigl\|\ektp\p_yu_\Phi\bigr\|_{\wt{L}^2_t(\cB^{\f52})};\\
\bigl\|\ektp(v\p_yu)_\Phi\bigr\|_{\wt{L}^2_t(\cB^{\f32})}\lesssim &\|u_\Phi\|_{\wt{L}^\infty_t(\cB^{\f12})}\bigl\|\ektp\p_yu_\Phi\bigr\|_{\wt{L}^2_t(\cB^{\f52})}+
\|u_\Phi\|_{\wt{L}^\infty_t(\cB^{\f52})}\bigl\|\ektp\p_yu_\Phi\bigr\|_{\wt{L}^2_t(\cB^{\f12})}.
\end{align*}
Inserting the above estimates into \eqref{S4eq13a} and then using \eqref{S1eq8a}, \eqref{S1eq12} and Proposition \ref{prop4.1},
we achieve
\beno
\bigl\|\ektp(\pa_tu)_\Phi\bigr\|_{\wt{L}^2_t(\cB^{\f32})}\lesssim \bigl\|e^{a|D_x|}\pa_yu_0\bigr\|_{\cB^{\f32}}+ \bigl\|e^{a|D_x|}u_0\bigr\|_{\cB^{\f52}}.
\eeno
This completes the proof of Theorem \ref{th1.2}.
\end{proof}

Now let us present the proof of the above two propositions.

\begin{proof}[Proof of Proposition \ref{prop4.1}]
We first deduce from Remark \ref{rmk3.1} that for any $s>0$
\beq \label{S4eq15}
\int_0^t\bigl|\bigl(\D_k^\h\left(u\p_xu\right)_\Phi | \D_k^\h u_\Phi\bigr)_{L^2}\bigr|\,dt'\lesssim d_k^2
2^{-2ks}\|u_\Phi\|_{\wt{L}^2_{t,\dot{\theta}(t)}(\cB^{s+\f12})}^2.
\eeq
While it follows from the proof of Lemma \ref{lem3.2} that
\beno
\int_0^t\bigl|\bigl(\D_k^\h(T^\h_{\p_y u}v+R^\h(v,\p_yu))_\Phi|\D_k^\h u_\Phi\bigr)_{L^2}\bigr|\,dt'\lesssim d_k^2
2^{-2ks}\|u_\Phi\|_{\wt{L}^2_{t,\dot{\theta}(t)}(\cB^{s+\f12})}^2.
\eeno
In view of \eqref{S3eq11}, we have
\beno
\|\D_k^\h v_\Phi(t)\|_{L^\infty}\lesssim d_k(t)2^{\f{k}2}\|u_\Phi(t)\|_{\cB^{\f32}}^{\f12}\|\p_yu_\Phi(t)\|_{\cB^{\f12}}^{\f12},
\eeno
so that there holds
\beq \label{S4eq20}
\begin{split}
\int_0^t&\bigl|\bigl(\D_k^{\rm h}(T^\h_{v}\p_yu)_\Phi\ |\ \D_k^{\rm
h}u_\Phi\bigr)_{L^2}\bigr|\,dt'\\
\lesssim & \sum_{|k'-k|\leq 4}\int_0^t\|S_{k'-1}^\h v_\Phi(t')\|_{L^\infty}
\|\D_{k'}^\h\p_yu_\Phi(t)\|_{L^2}\|\D_k^\h u_\Phi(t')\|_{L^2}\,dt'\\
\lesssim & \sum_{|k'-k|\leq 4}2^{\f{k'}2}\|u_\Phi\|_{L^\infty_t(\cB^{\f32})}^{\f12}
\|\D_{k'}^\h\p_yu_\Phi\|_{L^2_t(L^2)}\Bigl(\int_0^t\|\p_yu_\Phi(t')\|_{\cB^{\f12}}\|\D_k^\h u_\Phi(t')\|_{L^2}^2\,dt'\Bigr)^{\f12}\\
\lesssim &d_k^22^{-2ks}\|u_\Phi\|_{L^\infty_t(\cB^{\f32})}^{\f12}\|\p_yu_\Phi\|_{\wt{L}^2_t(\cB^s)}\|u_\Phi\|_{\wt{L}^2_{t,\dot{\theta}(t)}(\cB^{s+\f12})}.
\end{split}
\eeq
As a result, it comes out
\beq \label{S4eq16}
\begin{split}
\int_0^t\bigl|\bigl(\D_k^\h(v\p_y u)_\Phi&|\D_k^\h u_\Phi\bigr)_{L^2}\bigr|\,dt'\lesssim d_k^2
2^{-2ks}\|u_\Phi\|_{\wt{L}^2_{t,\dot{\theta}(t)}(\cB^{s+\f12})}\\
&\times\Bigl(\|u_\Phi\|_{\wt{L}^2_{t,\dot{\theta}(t)}(\cB^{s+\f12})}+\|u_\Phi\|_{\wt{L}^\infty_{t}(\cB^{\f32})}^\f12
\|\pa_yu_\Phi\|_{\wt{L}^2_{t}(\cB^{s})}\Bigr).
\end{split}\eeq

By virtue of \eqref{S4eq15} and \eqref{S4eq16}, we deduce from \eqref{S4eq6} that
\beno
\begin{split}
\f12\|&\ektp\D_k^{\rm h}u_\Phi\|_{L^\infty_t(L^2)}^2+\lam2^k\int_0^t\dot{\theta}(t')\|\ektp\D_k^{\rm
h}u_\Phi(t')\|_{L^2}^2\,dt'+\f12\|\ektp\D_k^\h\p_yu_\Phi\|_{L^2_t(L^2)}^2\\
\leq & \f12\bigl\|e^{a|D_x|}\D_k^\h u_0\bigr\|_{L^2}^2+Cd_k^2
2^{-2ks}\|\ektp u_\Phi\|_{\wt{L}^2_{t,\dot{\theta}(t)}(\cB^{s+\f12})}\\
&\qquad\qquad\qquad\qquad\times\Bigl(\|\ektp u_\Phi\|_{\wt{L}^2_{t,\dot{\theta}(t)}(\cB^{s+\f12})}+\|u_\Phi\|_{\wt{L}^\infty_{t}(\cB^{\f32})}^\f12
\|\ektp \pa_yu_\Phi\|_{\wt{L}^2_{t}(\cB^{s})}\Bigr),
\end{split} \eeno
from which, we infer
\beno
\begin{split}
\|&\ektp u_\Phi\|_{\wt{L}^\infty_t(\cB^{s})}+\sqrt{\lam}\|\ektp u_\Phi\|_{\wt{L}^2_{t,\dot\theta(t)}(\cB^{s+\f12})}
+\|\ektp\p_y u_\Phi\|_{\wt{L}^2_t(\cB^{s})}\leq C\Bigl(\bigl\|e^{a|D_x|}u_0\bigr\|_{\cB^{s}}\\
& \qquad\qquad
+\|\ektp u_\Phi\|_{\wt{L}^2_{t,\dot\theta(t)}(\cB^{s+\f12})}+\|u_\Phi\|_{\wt{L}^\infty_{t}(\cB^{\f32})}^{\f14}
\|\ektp\pa_yu_\Phi\|_{\wt{L}^2_{t}(\cB^{s})}^{\f12}\|\ektp u_\Phi\|_{\wt{L}^2_{t,\dot\theta(t)}(\cB^{s+\f12})}^{\f12}\Bigr).
\end{split}
\eeno
Applying Young's inequality yields
$$\longformule{
C\|u_\Phi\|_{\wt{L}^\infty_{t}(\cB^{\f32})}^{\f14}
\|\ektp\pa_yu_\Phi\|_{\wt{L}^2_{t}(\cB^{s})}^{\f12}\|\ektp u_\Phi\|_{\wt{L}^2_{t,\dot\theta(t)}(\cB^{s+\f12})}^{\f12}}{{}\leq
C\|u_\Phi\|_{\wt{L}^\infty_{t}(\cB^{\f32})}^{\f12}
\|\ektp u_\Phi\|_{\wt{L}^2_{t,\dot\theta(t)}(\cB^{s+\f12})}+\f12\|\ektp\pa_yu_\Phi\|_{\wt{L}^2_{t}(\cB^{s})}.}
$$
Therefore if we take \beq \label{S4eq18} \lam\geq C^2\bigl(1+\|u_\Phi\|_{\wt{L}^\infty_{t}(\cB^{\f32})}\bigr),\eeq
 we obtain
\beq \label{S4eq17}
\|\ektp u_\Phi\|_{\wt{L}^\infty_t(\cB^{s})}
+\|\ektp\p_y u_\Phi\|_{\wt{L}^2_t(\cB^{s})}\leq C\bigl\|e^{a|D_x|}u_0\bigr\|_{\cB^{s}}.
\eeq
which in particular implies that under the condition \eqref{S4eq18}, there holds
\beno
\|u_\Phi\|_{\wt{L}^\infty_t(\cB^{\f32})}
\leq C\bigl\|e^{a|D_x|}u_0\bigr\|_{\cB^{\f32}}.\eeno
Then by taking $\lam=C^2\bigl(1+\bigl\|e^{a|D_x|}u_0\bigr\|_{\cB^{\f32}}\bigr),$ \eqref{S4eq18} holds. Therefore
under the condition \eqref{S1eq14},  both \eqref{S3eq19} and \eqref{S4eq18} hold, and thus  \eqref{S4eq17} holds for any $t>0,$ which leads to
 \eqref{S4eq12}.
This completes the proof of the proposition.
\end{proof}

\begin{proof}[Proof of Proposition \ref{prop4.2}] Due to $\p_xu+\p_yv=0,$ we get, by applying $\p_y$ to \eqref{S1eq6}, that
\beno
\p_t\p_yu+u\p_x\p_yu+v\p_y^2u-\p_y^3u+\p_x\p_yp=0,
\eeno
from which, we get, by using a similar derivation of \eqref{S4eq6}, that
\beq \label{S4eq19}
\begin{split}
\f12\|&\ektp\D_k^{\rm h}\p_yu_\Phi\|_{L^\infty_t(L^2)}^2+\lam2^k\int_0^t\dot{\theta}(t')\|\ektp\D_k^{\rm
h}\p_yu_\Phi(t')\|_{L^2}^2\,dt'+\f12\|\ektp\D_k^\h\p^2_y u_\Phi\|_{L^2_t(L^2)}^2\\
\leq & \f12\bigl\|e^{a|D_x|}\D_k^\h \p_yu_0\bigr\|_{L^2}^2+\int_0^t\bigl|\bigl(\ektp\D_k^\h\left(u\p_x\p_yu\right)_\Phi | \ektp\D_k^\h \p_yu_\Phi\bigr)_{L^2}\bigr|\,dt'\\
&\qquad\qquad\qquad\qquad\qquad+\int_0^t\bigl|\bigl(\ektp\D_k^\h\left(v\p_y^2u\right)_\Phi | \ektp\D_k^\h \p_yu_\Phi\bigr)_{L^2}\bigr|\,dt'.
\end{split}
\eeq
It follows from the proof of Lemma \ref{lem3.1} that for any $s>0$
\beno
\int_0^t\bigl|\bigl(\D_k^\h\bigl(T^\h_u\p_x\p_yu+R^\h(u,\p_x\p_yu) \bigr)_\Phi | \D_k^\h \p_yu_\Phi\bigr)_{L^2}\bigr|\,dt'\lesssim
d_k^22^{-2ks}\|\p_yu_\Phi\|_{\wt{L}^2_{t,\dot{\theta}(t)}(\cB^{s+\f12})}.
\eeno
While we deduce from Lemma \ref{lem:Bern} and Definition \ref{def1.1} that
\beq \label{S4eq19a}
\begin{split}
\int_0^t&\bigl|\bigl(\D_k^{\rm h}(T^\h_{\p_x\p_yu}u)_\Phi\ |\ \D_k^{\rm
h}\p_y u_\Phi\bigr)_{L^2}\bigr|\,dt'\\
\lesssim & \sum_{|k'-k|\leq 4}\int_0^t\|S_{k'-1}^\h \p_x\p_yu_\Phi(t')\|_{L^\infty_\h(L^2_{\rm v})}
\|\D_{k'}^\h u_\Phi(t')\|_{L^2_\h(L^\infty_{\rm v})}\|\D_k^\h \p_y u_\Phi(t')\|_{L^2}\,dt'\\
\lesssim & \sum_{|k'-k|\leq 4}2^{k'}\int_0^t\|\p_y u_\Phi(t')\|_{\cB^{\f12}}\|\D_{k'}^\h \p_y u_\Phi(t')\|_{L^2}\|\D_k^\h \p_y u_\Phi(t')\|_{L^2}\,dt'\\
\lesssim & \sum_{|k'-k|\leq 4}2^{k'}\Bigl(\int_0^t\|\p_yu_\Phi(t')\|_{\cB^{\f12}}\|\D_{k'}^\h \p_yu_\Phi(t')\|_{L^2}^2\,dt'\Bigr)^{\f12}\\
&\qquad\qquad\qquad\qquad\times\Bigl(\int_0^t\|\p_yu_\Phi(t')\|_{\cB^{\f12}}\|\D_k^\h \p_yu_\Phi(t')\|_{L^2}^2\,dt'\Bigr)^{\f12}\\
\lesssim &d_k^22^{-2ks}\|\p_yu_\Phi\|_{\wt{L}^2_{t,\dot{\theta}(t)}(\cB^{s+\f12})}^2.
\end{split}
\eeq
As a result,  it comes out that for any $s>0,$
\beq \label{S4eq21}
\begin{split}
\int_0^t\bigl|\bigl(\D_k^\h(u\p_x\p_yu)_\Phi |& \D_k^\h \p_yu_\Phi\bigr)_{L^2}\bigr|\,dt'\lesssim
d_k^22^{-2ks}\|\p_yu_\Phi\|_{\wt{L}^2_{t,\dot{\theta}(t)}(\cB^{s+\f12})}^2.
\end{split}
\eeq
On the other hand, we deduce from  Lemma \ref{lem:Bern} and \eqref{S3eq11} that for any $s>0$
\begin{align*}
\int_0^t\bigl|&\bigl(\D_k^\h\bigl(R^\h(v,\p_y^2u) \bigr)_\Phi | \D_k^\h \p_yu_\Phi\bigr)_{L^2}\bigr|\,dt'\\
\lesssim &2^{\f{k}2}\sum_{k'\geq k-3}\int_0^t\|\D_{k'}^\h v_\Phi(t')\|_{L^2_\h(L^\infty_{\rm v})}\|\wt{\D}_{k'}^\h \p_y^2u_\Phi(t')\|_{L^2}
\|{\D}_{k'}^\h \p_yu_\Phi(t')\|_{L^2}\,dt'\\
\lesssim &2^{\f{k}2}\sum_{k'\geq k-3}\int_0^t\|u_\Phi(t')\|_{\cB^{\f32}}^{\f12}\|\p_y u_\Phi(t')\|_{\cB^{\f12}}^{\f12}\|\wt{\D}_{k'}^\h \p_y^2u_\Phi(t')\|_{L^2}
\|{\D}_{k'}^\h \p_yu_\Phi(t')\|_{L^2}\,dt'\\
\lesssim &2^{\f{k}2}\sum_{k'\geq k-3}\|u_\Phi\|_{L^\infty_t(\cB^{\f32})}^{\f12}\|\wt{\D}_{k'}^\h \p_y^2u_\Phi\|_{L^2_t(L^2)}
\Bigl(\int_0^t\|\p_y u_\Phi(t')\|_{\cB^{\f12}}
\|{\D}_{k'}^\h \p_yu_\Phi(t')\|_{L^2}^2\,dt'\Bigr)^{\f12}\\
\lesssim &
d_k^22^{-2ks}\|u_\Phi\|_{L^\infty_t(\cB^{\f32})}^{\f12}\|\p_y^2u_\Phi\|_{\wt{L}^2_{t}(\cB^{s})}\|\p_yu_\Phi\|_{\wt{L}^2_{t,\dot{\theta}(t)}(\cB^{s+\f12})}.
\end{align*}
And the proof of \eqref{S4eq20} ensures that
\beno
\begin{split}
\int_0^t\bigl|\bigl(\D_k^{\rm h}(T^\h_{v}\p_y^2u)_\Phi\ |&\ \D_k^{\rm
h}\p_yu_\Phi\bigr)_{L^2}\bigr|\,dt'\\
\lesssim &d_k^22^{-2ks}\|u_\Phi\|_{L^\infty_t(\cB^{\f32})}^{\f12}\|\p_y^2u_\Phi\|_{\wt{L}^2_t(\cB^s)}\|\p_yu_\Phi\|_{\wt{L}^2_{t,\dot{\theta}(t)}(\cB^{s+\f12})}.
\end{split}
\eeno
Finally, by using integration by parts, we have
\begin{align*}
\int_0^t\bigl|\bigl(\D_k^\h\bigl(T^\h_{\p_y^2u}v\bigr)_\Phi | \D_k^\h \p_yu_\Phi\bigr)_{L^2}\bigr|\,dt'\le&
\int_0^t\bigl|\bigl(\D_k^\h\bigl(T^\h_{\p_yu}\pa_yv\bigr)_\Phi | \D_k^\h \p_yu_\Phi\bigr)_{L^2}\bigr|\,dt\\&+
\int_0^t\bigl|\bigl(\D_k^\h\bigl(T^\h_{\p_yu}v\bigr)_\Phi | \D_k^\h \p_y^2u_\Phi\bigr)_{L^2}\bigr|\,dt.
\end{align*}
Due to $\p_xu+\p_yv=0,$ we deduce from a similar derivation of \eqref{S4eq19a} that
\begin{align*}
\int_0^t\bigl|&\bigl(\D_k^\h\bigl(T^\h_{\p_yu}\pa_yv\bigr)_\Phi | \D_k^\h \p_yu_\Phi\bigr)_{L^2}\bigr|\,dt\\
\lesssim& \sum_{|k'-k|\leq 4}\int_0^t\|S_{k'-1}^\h\p_yu_\Phi(t')\|_{L^\infty_\h(L^2_{\rm v})}\|\D_{k'}^\h \p_xu_\Phi(t')\|_{L^2_\h(L^\infty_{\rm v})}
\|\D_k^\h\p_y u_\Phi(t')\|_{L^2}\,dt'\\
\lesssim& \sum_{|k'-k|\leq 4}2^{k'}\int_0^t\|\p_y u_\Phi(t')\|_{\cB^{\f12}}\|\D_{k'}^\h \p_y u_\Phi(t')\|_{L^2}
\|\D_k^\h\p_y u_\Phi(t')\|_{L^2}\,dt'\\
\lesssim&  d_k^22^{-2ks}\|\p_yu_\Phi\|_{\wt{L}^2_{t,\dot{\theta}(t)}(\cB^{s+\f12})}^2.
\end{align*}
While we observe that
\begin{align*}
\int_0^t\bigl|&\bigl(\D_k^\h\bigl(T^\h_{\p_yu}v\bigr)_\Phi | \D_k^\h \p_y^2u_\Phi\bigr)_{L^2}\bigr|\,dt\\
\lesssim&  \sum_{|k'-k|\leq 4}\int_0^t\|S_{k'-1}^\h\p_yu_\Phi(t')\|_{L^\infty_\h(L^2_{\rm v})}\|\D_{k'}^\h v_\Phi(t')\|_{L^2_\h(L^\infty_{\rm v})}
\|\D_k^\h\p_y^2 u_\Phi(t')\|_{L^2}\,dt'\\
\lesssim& \sum_{|k'-k|\leq 4}2^{k'}\int_0^t\|\p_y u_\Phi(t')\|_{\cB^{\f12}}\|\D_{k'}^\h  u_\Phi(t')\|_{L^2}
\|\D_k^\h\p_y^2 u_\Phi(t')\|_{L^2}\,dt'\\
\lesssim&  d_k^22^{-2ks}\|\pa_yu_\Phi\|_{\wt{L}^2_{t}(\cB^{\f12})}
\|u_\Phi\|_{{L}^\infty_{t}(\cB^{s+1})}\|\pa_y^2u_\Phi\|_{\wt{L}^2_{t}(\cB^{s})}.
\end{align*}
This gives rise to
\begin{align*}
\int_0^t\bigl|\bigl(\D_k^\h\bigl(T^\h_{\p_y^2u}v\bigr)_\Phi | \D_k^\h \p_yu_\Phi\bigr)_{L^2}\bigr|\,dt'\lesssim &
d_k^22^{-2ks}\Bigl(\|\p_yu_\Phi\|_{\wt{L}^2_{t,\dot{\theta}(t)}(\cB^{s+\f12})}^2\\
&+\|\pa_yu_\Phi\|_{\wt{L}^2_{t}(\cB^{\f12})}
\|u_\Phi\|_{{L}^\infty_{t}(\cB^{s+1})}\|\pa_y^2u_\Phi\|_{\wt{L}^2_{t}(\cB^{s})}\Bigr).
\end{align*}

By summarizing the above estimates, we obtain
\beq \label{S4eq22}
\begin{split}
\int_0^t\bigl|\bigl(\D_k^\h(&v\p_y^2u)_\Phi | \D_k^\h \p_yu_\Phi\bigr)_{L^2}\bigr|\,dt'\\
\lesssim &
d_k^22^{-2ks}\Bigl(\|u_\Phi\|_{L^\infty_t(\cB^{\f32})}^{\f12}\|\p_yu_\Phi\|_{\wt{L}_{t,\dot{\theta}(t)}(\cB^{s+\f12})}
\|\p_y^2u_\Phi\|_{\wt{L}^2_t(\cB^s)}\\
&\qquad+\|\p_yu_\Phi\|_{\wt{L}_{t,\dot{\theta}(t)}(\cB^{s+\f12})}^2+\|\pa_yu_\Phi\|_{\wt{L}^2_{t}(\cB^{\f12})}
\|u_\Phi\|_{{L}^\infty_{t}(\cB^{s+1})}\|\pa_y^2u_\Phi\|_{\wt{L}^2_{t}(\cB^{s})}\Bigr).
\end{split}
\eeq

By inserting \eqref{S4eq21} and \eqref{S4eq22} into \eqref{S4eq19} and then repeating the last step of the
proof of Proposition \ref{prop4.1}, we obtain
\beno
\begin{split}
&\|\ektp \p_y u_\Phi\|_{\wt{L}^\infty_t(\cB^{s})}+\sqrt{\lam}\|\ektp \p_y u_\Phi\|_{\wt{L}^2_{t,\dot\theta(t)}(\cB^{s+\f12})}
+\|\ektp\p_y^2 u_\Phi\|_{\wt{L}^2_t(\cB^{s})}\\
& \ \leq \bigl\|e^{a|D_x|}\p_yu_0\bigr\|_{\cB^{s}}
+C\Bigl(\|\ektp u_\Phi\|_{\wt{L}^2_{t,\dot\theta(t)}(\cB^{s+\f12})}
+\bigl(\|u_\Phi\|_{L^\infty_t(\cB^{\f32})}^{\f14}\|\ektp\p_yu_\Phi\|_{\wt{L}_{t,\dot{\theta}(t)}(\cB^{s+\f12})}^{\f12}\\
&\qquad\qquad\qquad\qquad\qquad\qquad\qquad+\|\pa_yu_\Phi\|_{\wt{L}^2_{t}(\cB^{\f12})}^{\f12}
\|\ektp u_\Phi\|_{{L}^\infty_{t}(\cB^{s+1})}^{\f12}\bigr)
\|\ektp\p_y^2u_\Phi\|_{\wt{L}^2_t(\cB^s)}^{\f12}\Bigr).
\end{split}
\eeno
Applying Young's inequality yields
\beno
\begin{split}
&\|\ektp \p_y u_\Phi\|_{\wt{L}^\infty_t(\cB^{s})}+\sqrt{\lam}\|\ektp \p_y u_\Phi\|_{\wt{L}^2_{t,\dot\theta(t)}(\cB^{s+\f12})}
+\|\ektp\p_y^2 u_\Phi\|_{\wt{L}^2_t(\cB^{s})}\\
& \ \leq \bigl\|e^{a|D_x|}\p_yu_0\bigr\|_{\cB^{s}}
+C\Bigl(\bigl(1+\|u_\Phi\|_{L^\infty_t(\cB^{\f32})}^{\f12}\bigr)\|\ektp u_\Phi\|_{\wt{L}^2_{t,\dot\theta(t)}(\cB^{s+\f12})}
\\
&\qquad\qquad\qquad\qquad\qquad+\|\pa_yu_\Phi\|_{\wt{L}^2_{t}(\cB^{\f12})}
\|\ektp u_\Phi\|_{{L}^\infty_{t}(\cB^{s+1})}\bigr)+\f12
\|\ektp\p_y^2u_\Phi\|_{\wt{L}^2_t(\cB^s)},
\end{split}
\eeno
from which, \eqref{S1eq8a}, \eqref{S1eq12} and Proposition \ref{prop4.1}, we infer
\begin{align*}
&\|\ektp \p_y u_\Phi\|_{\wt{L}^\infty_t(\cB^{s})}+\sqrt{\lam}\|\ektp \p_y u_\Phi\|_{\wt{L}^2_{t,\dot\theta(t)}(\cB^{s+\f12})}
+\|\ektp\p_y^2 u_\Phi\|_{\wt{L}^2_t(\cB^{s})}\\
& \ \leq \bigl\|e^{a|D_x|}\p_yu_0\bigr\|_{\cB^{s}}
+C\Bigl(\bigl(1+\bigl\|e^{a|D_x|}u_0\bigr\|_{\cB^{\f32}}^{\f12}\bigr)\|\ektp u_\Phi\|_{\wt{L}^2_{t,\dot\theta(t)}(\cB^{s+\f12})}
+\bigl\|e^{a|D_x|}u_0\bigr\|_{\cB^{s+1}}\Bigr).
\end{align*}
Taking $\lam=C^2\bigl(1+\bigl\|e^{a|D_x|}u_0\bigr\|_{\cB^{\f32}}\bigr)$ in the above inequality leads to
 \eqref{S4eq13}.
This completes the proof of Proposition \ref{prop4.2}.
\end{proof}

\renewcommand{\theequation}{\thesection.\arabic{equation}}
\setcounter{equation}{0}
\section{The Convergence to the hydrostatic Navier-Stokes system}

In this section, we justify the limit from the scaled anisotropic Navier-Stokes system
 to the hydrostatic Navier-Stokes system in a 2-D striped domain. To this end, we introduce
\beno
w^1_\e\eqdefa u^\e-u,\quad w^2_\e\eqdefa v^\e-v,\quad q_\e\eqdefa p^\e-p.
\eeno
Then $(w^1_\e,w^2_\e,q_\e)$ verifies
\begin{equation}\label{S5eq1}
 \quad\left\{\begin{array}{l}
\displaystyle \p_tw^1_\e-\e^2\p_x^2w^1_\e-\p_y^2w^1_\e+\p_xq_\e=R^1_\e\ \ \mbox{in} \ \cS\times ]0,\infty[,\\
\displaystyle \e^2\left(\p_tw^2_\e-\e^2\p_x^2w^2_\e-\p_y^2w^2_\e\right)+\p_yq_\e=R^2_\e,\\
\displaystyle \p_xw^1_\e+\p_yw^2_\e=0,\\
\displaystyle \left(w^1_\e,w^2_\e\right)|_{y=0}=\left(w^1_\e,w^2_\e\right)|_{y=1}=0,\\
\displaystyle \left(w^1_\e, w^2_\e\right)|_{t=0}=\left(u_0^\e-u_0, v_0^\e-v_0\right),
\end{array}\right.
\end{equation}
where $v_0$ is determined from $u_0$ via $\p_xu_0+\p_yv_0=0$ and $v_0|_{y=0}=v_0|_{y=1}=0,$ and
\beq\label{S5eq2}
\begin{split}
&R^1_\e=\e^2\p_x^2u-\big(u^\e\p_xu^\e-u\p_xu\big)-\big(v^\e\p_yu^\e-v\p_yu\big),\\
&R^2_\e=-\e^2\big(\p_tv-\e^2\p_x^2v-\p_y^2v+u^\e\p_xv^\e+v^\e\p_yv^\e\big).
\end{split} \eeq
Let us define
\beq \label{theta} u_{\Th}(t,x,y)\eqdefa\cF_{\xi\to
x}^{-1}\bigl(e^{\Th(t,\xi)}\widehat{u}^(t,\xi,y)\bigr)\andf
\Th(t,\xi)\eqdefa \big(a-\mu\zeta(t)\big)|\xi|,
\eeq
where $\mu\ge \lambda$ will be determined later, and ${\zeta}(t)$ is given by
\begin{equation}\nonumber
{\zeta}(t)=\int_0^t\bigl(\bigl\|\left(\p_yu^\e_{\Psi}, \e \p_xu^\e_{\Psi}\right)(t')\bigr\|_{\cB^{\f12}}
+\|\p_yu_\Phi(t')\|_{\cB^{\f12}}\bigr)\,dt'.
\end{equation} 
Similar notation for $(w^1_\e)_\Th$ and so on.

It is easy to observe that if we take $c_0$ in \eqref{S1eq8} and $c_1$ in \eqref{S1eq8a} small enough, then $\Th(t)\ge 0$
and
\beno
\Th(t,\xi)\leq \min\left( \Psi(t,\xi), \Phi(t,\xi)\right).
\eeno
Thanks to Theorem \ref{th1.2}, we deduce that
\beq \label{S5eq3} \|u^\e_{\Psi}\|_{\wt{L}^\infty(\R^+;\cB^{\f12})}+
\|u_\Phi\|_{\wt{L}^\infty(\R^+;\cB^\f12\cap\cB^{\f52})}+\|\p_y u_\Phi\|_{\wt{L}^2(\R^+;\cB^\f12\cap\cB^{\f52})}+\|(\pa_tu)_\Phi\|_{\wt{L}^\infty(\R^+;\cB^{\f32})}\leq M,
\eeq
where $u^\e_\Psi$ and $u_\Phi$ are determined respectively  by \eqref{S3eq1} and 
\eqref{S4eq1} and $M\ge 1$ is a constant independent of  $\e$. \smallskip

In what follows, we shall neglect the subscript $\e$ in $(w^1_\e,w^2_\e).$

\begin{proof}[Proof of Theorem \ref{thm3}]
In view of \eqref{S5eq1}, we get, by using a similar derivation of \eqref{S3eq6}, that
\beq \label{S5eq6}
\begin{split}
\|\D_k^{\rm h}&(w^1_\Th, \e w^2_\Th)\|_{L^\infty_t(L^2)}^2+\mu 2^k\int_0^t\dot{\zeta}(t')\|\D_k^{\rm
h}(w^1_\Th,\e w^2_\Th)(t')\|_{L^2}^2\,dt'\\
&+\int_0^t\bigl(\|\D_k^\h\p_y (w^1_\Th,\e w^2_\Th)(t')\|_{L^2}^2+\e^2 2^{2k}\|\D_k^\h (w^1_\Th, \e w^2_\Th)(t')\|_{L^2}^2\bigr)\,dt'\\
\leq & \bigl\|e^{a|D_x|}\D_k^\h(u_0^\e-u_0,\e (v_0^\e-v_0))\bigr\|_{L^2}^2\\
&+\int_0^t\bigl|\bigl(\D_k^\h R^1_\Th | \D_k^\h w^1_\Th\bigr)_{L^2}\bigr|\,dt'
+\int_0^t\bigl|\bigl(\D_k^\h R^2_\Th | \D_k^\h w^2_\Th\bigr)_{L^2}\bigr|\,dt'.
\end{split}
\eeq
We now claim that
\beq \label{S5eq7}
\begin{split}
\int_0^t\bigl|\bigl(\D_k^\h R^1_\Th |& \D_k^\h w^1_\Th\bigr)_{L^2}\bigr|\,dt'\lesssim d_k^22^{-k}\Bigl(\e\| \p_yu_\Th\|_{\wt{L}^2_t(\cB^{\f32})}
\|\e w^1_\Th\|_{\wt{L}^2_t(\cB^{\f32})}\\
&+\|u_\Th\|^{\f12}_{L^\infty_t(\cB^{\f32})}
\|\p_yw^1_\Th\|_{\wt{L}^2_t(\cB^{\f12})}\|w^1_\Th\|_{\wt{L}^2_{t,\dot\zeta(t)}(\cB^{1})}+\|w^1_\Th\|_{\wt{L}^2_{t,\dot\zeta(t)}(\cB^{1})}^2\Bigr),
\end{split}
\eeq
and
\beq \label{S5eq8}
\begin{split}
\int_0^t\bigl|\bigl(\D_k^\h R^2_\Th | \D_k^\h& w^2_\Th\bigr)_{L^2}\bigr|\,dt'\lesssim  d_k^22^{-k}\Bigl\{\bigl\|\left(w^1_\Th, \e w^2_\Th\right)\bigr\|_{\wt{L}^2_{t,\dot\zeta(t)}(\cB^{1})}^2+\e^2
\bigl\|(\p_yw^2_\Th,\e\p_xw^2_\Th)\|_{\wt{L}^2_t(\cB^{\f12})}\\
&\times \Bigl(\|(\p_tu)_\Th\|_{\wt{L}^2_t(\cB^{\f32})}
+\| \p_yu_\Th\|_{\wt{L}^2_t(\cB^{\f32})}+
\e\| \p_yu_\Th\|_{\wt{L}^2_t(\cB^{\f52})}\Bigr)
\\
&+\e^2\|w^2_\Th\|_{\wt{L}^2_{t,\dot\zeta(t)}(\cB^{1})}\Bigl(\|w^2_\Th\|_{\wt{L}^2_{t,\dot\zeta(t)}(\cB^{1})}+\|u^\e_\Th\|^{\f12}_{L^\infty_t(\cB^{\f12})}\| \p_yu_\Th\|_{\wt{L}^2_t(\cB^{2})}\\
&\qquad\qquad\qquad\qquad\ \ +\|u_\Th\|^{\f12}_{L^\infty_t(\cB^{\f32})}\bigl(\| \p_yw^2_\Th\|_{\wt{L}^2_t(\cB^{\f12})}+\|\p_y u_\Th\|_{\wt{L}^2_t(\cB^{\f32})}\bigr)\Bigr)\Bigr\}.
\end{split}
\eeq

By virtue of \eqref{S5eq3}, \eqref{S5eq7} and \eqref{S5eq8}, we infer
\beno
\begin{split}
\sum_{i=1}^2\int_0^t\bigl|\bigl(\D_k^\h R^i_\Th | \D_k^\h w^i_\Th\bigr)_{L^2}\bigr|\,dt'\lesssim  & d_k^22^{-k}\Bigl(
M\e \bigl\|\bigl(\e\p_x(w^1_\Th,\e w^2_\Th),\e
\p_yw^2_\Th\bigr)\bigr\|_{\wt{L}^2_t(\cB^{\f12})}
\\
&+M^{\f12}\bigl\|\p_y(w^1_\Th,\e w^2_\Th)\bigr\|_{\wt{L}^2_t(\cB^{\f12})}\bigl\|(w^1_\Th,\e w^2_\Th)\bigr\|_{\wt{L}^2_{t,\dot\zeta(t)}(\cB^{1})}\\
&+M^{\f32}\e\|\e w^2_\Th\|_{\wt{L}^2_{t,\dot\zeta(t)}(\cB^{1})}+\bigl\|(w^1_\Th,\e w^2_\Th)\bigr\|_{\wt{L}^2_{t,\dot\zeta(t)}(\cB^{1})}^2\Bigr),
\end{split}
\eeno
from which and \eqref{S5eq6}, we deduce that
\beq\label{eq:error}
\begin{split}
&\|(w^1_\Th,\e w^2_\Th)\|_{\wt{L}^\infty_t(\cB^{\f12})}+\mu^{\f12}\|(w^1_\Th,\e w^2_\Th)\|_{\wt{L}^2_{t,\dot\zeta_1(t)}(\cB^{1})}
+\|\p_y(w^1_\Th,\e w^2_\Th)\|_{\wt{L}^2_t(\cB^{\f12})}\\
&+\e\|(w^1_\Th,\e w^2_\Th)\|_{\wt{L}^2_t(\cB^{\f32})}\leq C\bigl\|e^{a|D_x|}\left(u_0^\e-u_0,\e (v_0^\e-v_0)\right)\bigr\|_{\cB^{\f12}}\\
&\qquad\qquad\qquad\qquad\qquad\qquad+C\Bigl(
\sqrt{M\e} \bigl\|\bigl(\e\p_x(w^1_\Th,\e w^2_\Th),\e
\p_yw^2_\Th\bigr)\bigr\|_{\wt{L}^2_t(\cB^{\f12})}^{\f12}
\\
&\qquad\qquad\qquad\qquad\qquad\qquad+M^{\f14}\bigl\|\p_y(w^1_\Th,\e w^2_\Th)\bigr\|_{\wt{L}^2_t(\cB^{\f12})}^{\f12}\bigl\|(w^1_\Th,\e w^2_\Th)\bigr\|_{\wt{L}^2_{t,\dot\zeta(t)}(\cB^{1})}^{\f12}\\
&\qquad\qquad\qquad\qquad\qquad\qquad+M^{\f34}\e^{\f12}\|\e w^2_\Th\|_{\wt{L}^2_{t,\dot\zeta(t)}(\cB^{1})}^{\f12}+\bigl\|(w^1_\Th,\e w^2_\Th)\bigr\|_{\wt{L}^2_{t,\dot\zeta(t)}(\cB^{1})}\Bigr).
\end{split}
\eeq
Applying Young's inequality gives rise to
\beno
\begin{split}
&\|(w^1_\Th,\e w^2_\Th)\|_{\wt{L}^\infty_t(\cB^{\f12})}+\mu^{\f12}\|(w^1_\Th,\e w^2_\Th)\|_{\wt{L}^2_{t,\dot\zeta(t)}(\cB^{1})}
+\|\p_y(w^1_\Th,\e w^2_\Th)\|_{\wt{L}^2_t(\cB^{\f12})}\\
&\qquad+\e^2\|(w^1_\Th,\e w^2_\Th)\|_{\wt{L}^2_t(\cB^{\f32})}\\
&\leq C\Bigl(\bigl\|e^{a|D_x|}\left(u_0^\e-u_0,\e (v_0^\e-v_0)\right)\bigr\|_{\cB^{\f12}}+
M\bigl(\e+
\bigl\|(w^1_\Th,\e w^2_\Th)\bigr\|_{\wt{L}^2_{t,\dot\zeta(t)}(\cB^{1})}\bigr)\Bigr).
\end{split}
\eeno
Taking $\mu=C^2M^2$ leads to \eqref{S1eq14}. This completes the proof of the theorem. \end{proof}

Now let us present the proof of \eqref{S5eq7} and \eqref{S5eq8}.

\begin{proof}[Proof of \eqref{S5eq7}]
According \eqref{S5eq2}, we write
\beno
R^1_\e=\e^2\p_x^2u-\big(u^\e\p_xw^1+w^1\p_xu\big)-\big(v^\e\p_yw^1+w^2\p_yu\big).
\eeno

We first observe that
\beq \label{S5eq9}
\e^2\int_0^t\big|\big(\Delta_k^\h\p_x^2u_{\Th}|\Delta_k^\h w^1_\Th\big)_{L^2}\big|\,dt'\leq Cd_k^22^{-k}\e\|\pa_y u_\Th\|_{\wt{L}^2_{t}(\cB^{\f32})}\|\e w^1_\Th\|_{\wt{L}^2_{t}(\cB^{\f32})}.
\eeq

\noindent$\bullet$\underline{
The estimate of $\int_0^t\big|\big(\Delta_k^\h(u^\e\p_xw^1+w^1\p_xu)_\Th | \Delta_k^\h w^1_\Th\big)_{L^2}\big|\,dt'$.}

It follows from Lemma \ref{lem3.1} that
\beq\label{S5eq10}
\int_0^t\big|\big(\Delta_k^\h(u^\e\p_xw^1)_{\Th}|\Delta_k^\h w^1_\Th\big)_{L^2}\big|\,dt'\lesssim d_k^22^{-k} \|w^1_\Th\|_{\wt{L}^2_{t,\dot{\zeta}(t)}(\cB^{1})}^2.
\eeq
By applying Bony's decomposition \eqref{Bony}   for the horizontal variable to $w^1\p_xu$, we obtain
\beno
w^1\p_xu=T^\h_{w^1}\p_xu+T^\h_{\p_xu}w^1+R^h(w^1,\p_xu).
\eeno
Notice that
\beno
\|\D_{k'}^\h\p_x u_\Th(t')\|_{L^2_\h(L^\infty_{\rm v})}\lesssim d_{k'}(t)\|u_\Th(t')\|_{\cB^{\f32}}^{\f12}\|\p_y u_\Th(t')\|_{\cB^{\f12}}^{\f12},
\eeno
we infer
\beno
\begin{split}
\int_0^t&\bigl|\bigl(\D_k^{\rm h}(T^\h_{w^1}\p_xu)_\Th\ |\ \D_k^{\rm
h}w_\Th^1\bigr)_{L^2}\bigr|\,dt'\\
\lesssim & \sum_{|k'-k|\leq 4}\int_0^t\|S_{k'-1}^\h w^1_\Th(t')\|_{L^\infty_\h(L^2_{\rm v})}
\|\D_{k'}^\h\p_xu_\Th(t')\|_{L^2_\h(L^\infty_{\rm v})}\|\D_k^\h w^1_\Th(t')\|_{L^2}\,dt'\\
\lesssim & \sum_{|k'-k|\leq 4}d_{k'}\|u_\Th\|_{L^\infty_t(\cB^{\f32})}^{\f12}\|S_{k'-1}^\h w^1_\Th\|_{L^2_t(L^\infty_\h(L^2_{\rm v}))}\Bigl(\int_0^t
\|\p_y u_\Th(t')\|_{\cB^{\f12}}\|\D_k^\h w^1_\Th(t')\|_{L^2}^2\,dt'\Bigr)^{\f12}\\
\lesssim & d_k^22^{-k}\|u_\Th\|^{\f12}_{L^\infty_t(\cB^{\f32})}
\|\p_yw^1_\Th\|_{\wt{L}^2_t(\cB^{\f12})}\|w^1_\Th\|_{\wt{L}^2_{t,\dot\zeta(t)}(\cB^{1})}.
\end{split}
\eeno
While observing that
\beno
\begin{split}
\|S_{k'-1}^\h\p_xu_\Th(t')\|_{L^\infty}\lesssim &\sum_{\ell\leq k'-2}2^{\f{3\ell}2}\|\D_\ell^\h u_\Th(t')\|_{L^2}^{\f12}\|\D_\ell^\h \p_y u_\Th(t')\|_{L^2}^{\f12}\\
\lesssim& d_{k'}(t)2^{k'}\|\p_y u_\Th(t')\|_{\cB^{\f12}},
\end{split}
\eeno
we deduce
\beno
\begin{split}
\int_0^t&\bigl|\bigl(\D_k^{\rm h}(T^\h_{\p_xu}w^1)_\Th\ |\ \D_k^{\rm
h}w^1_\Th\bigr)_{L^2}\bigr|\,dt'\\
\lesssim & \sum_{|k'-k|\leq 4}\int_0^t\|S_{k'-1}^\h \p_xu_\Th(t')\|_{L^\infty}
\|\D_{k'}^\h w^1_\Th(t')\|_{L^2}\|\D_k^\h w^1_\Th(t')\|_{L^2}\,dt'\\
\lesssim & \sum_{|k'-k|\leq 4}2^{{k'}}\Bigl(\int_0^t\|\D_{k'}^\h w_\Th^1(t')\|_{L^2}^2\|\p_yu_\Th(t')\|_{\cB^{\f12}}\,dt'\Bigr)^{\frac12}\\
&\qquad\qquad\qquad\qquad\qquad\times \Bigl(\int_0^t\|\D_k^\h w^1_\Th(t')\|_{L^2}^2\|\p_yu_\Th(t')\|_{\cB^{\f12}}\,dt'\Bigr)^{\frac12}\\
\lesssim & d_k^22^{-k}\|w^1_\Th\|_{\wt{L}^2_{t,\dot\zeta(t)}(\cB^{1})}^2.
\end{split}
\eeno
Along the same line, we have
\beno
\begin{split}
\int_0^t\bigl|&\bigl(\D_k^{\rm h}(R^\h(w^1,\p_xu)_\Th\ |\ \D_k^{\rm
h}w^1_\Th\bigr)_{L^2}\bigr|\,dt'\\
\lesssim &2^{\f{k}2}\sum_{k'\geq k-3}\int_0^t\|{\D}_{k'}^\h w^1_\Th(t')\|_{L^2}\|\wt{\D}_{k'}^\h\p_x u_\Th(t')\|_{L^2_\h(L^\infty_{\rm v})}\|\D_k^\h w^1_\Th(t')\|_{L^2}\,dt'\\
\lesssim & 2^{\f{k}2}\sum_{k'\geq k-3}2^{\f{k'}2}\Bigl(\int_0^t\|{\D}_{k'}^\h  w^1_\Th(t')\|_{L^2}^2\|\p_yu_\Th(t')\|_{\cB^{\f12}}\,dt'\Bigr)^{\f12}
\\
&\qquad\qquad\qquad\qquad\times \Bigl(\int_0^t\|\D_k^\h w^1_\Th(t')\|_{L^2}^2\|\p_yu_\Th(t')\|_{\cB^{\f12}}\,dt'\Bigr)^{\f12}\\
\lesssim &d_k^22^{-k}\|w^1_\Th\|_{\wt{L}^2_{t,\dot{\zeta}(t)}(\cB^{1})}^2.
\end{split}
\eeno
As a result, it comes out
\beq \label{S5eq11}
\begin{split}
\int_0^t\bigl|\bigl(\D_k^{\rm h}&({w^1}\p_xu)_\Th\ |\ \D_k^{\rm
h}w_\Th^1\bigr)_{L^2}\bigr|\,dt'\\
\lesssim &d_k^22^{-k}\|w_\Th^1\|_{\wt{L}^2_{t,\dot{\zeta}(t)}(\cB^{1})}\Bigl(\|w^1_\Th\|_{\wt{L}^2_{t,\dot{\zeta}(t)}(\cB^{1})}
+\|u_\Th\|^{\f12}_{L^\infty_t(\cB^{\f32})}
\|\p_yw^1_\Th\|_{\wt{L}^2_t(\cB^{\f12})}\Bigr).
\end{split}
\eeq

\noindent$\bullet$\underline{
The estimate of $\int_0^t\big|\big(\Delta_k^\h(v^\e\p_yw^1)_\Th | \Delta_k^\h w^1_\Th\big)_{L^2}\big|\,dt'$.}

We write
\beno
v^\e\p_y w^1=w^2\p_yw^1+v\p_yw^1.
\eeno
We first deduce from Lemma \ref{lem3.2} that
\beq\label{S5eq11a}
\int_0^t\big|\big(\Delta_k^\h(w^2\p_yw^1)_\Th | \Delta_k^\h w^1_\Th\big)_{L^2}\big|\,dt'\lesssim d_k^22^{-k}\|w_\Th^1\|_{\wt{L}^2_{t,\dot{\zeta}(t)}(\cB^{1})}^2.
\eeq

Whereas by applying Bony's decomposition \eqref{Bony}   for the horizontal variable to $v\p_xw^1$, we find
\beno
v\p_y w^1=T^\h_{v}\p_y w^1+T^\h_{\p_yw^1}v+R^h(v,\p_y w^1).
\eeno
It follows from \eqref{S3eq8} that
\beno
\begin{split}
\|S_{k'-1}^\h v_\Th(t')\|_{L^\infty}\lesssim& \sum_{\ell\leq k'-2}2^{\f{3\ell}2}\|\D_{\ell}^\h u_\Th(t')\|_{L^2}^{\f12}
\|\D_{\ell}^\h \p_y u_\Th(t')\|_{L^2}^{\f12}\\
\lesssim &
d_{k'}(t)2^{\f{k'}2}\|u_\Th(t')\|_{\cB^{\f32}}^{\f12}\|\p_y u_\Th(t')\|_{\cB^{\f12}}^{\f12},
\end{split}
\eeno
from which, we infer
\beno
\begin{split}
\int_0^t&\bigl|\bigl(\D_k^{\rm h}(T^\h_{v}\p_y w^1)_\Th\ |\ \D_k^{\rm
h}w_\Th^1\bigr)_{L^2}\bigr|\,dt'\\
\lesssim & \sum_{|k'-k|\leq 4}\int_0^t\|S_{k'-1}^\h v_\Th(t')\|_{L^\infty}
\|\D_{k'}^\h\p_y w^1_\Th(t')\|_{L^2}\|\D_k^\h w^1_\Th(t')\|_{L^2}\,dt'\\
\lesssim & \sum_{|k'-k|\leq 4}2^{\f{k'}2}\|u_\Th\|_{L^\infty_t(\cB^{\f32})}^{\f12}\|\D_{k'}^\h\p_y w^1_\Th(t')\|_{L^2_t(L^2)}
\Bigl(\int_0^t
\|\p_y u_\Th(t')\|_{\cB^{\f12}}\|\D_k^\h w^1_\Th(t')\|_{L^2}^2\,dt'\Bigr)^{\f12}\\
\lesssim & d_k^22^{-k}\|u_\Th\|^{\f12}_{L^\infty_t(\cB^{\f32})}
\|\p_yw^1_\Th\|_{\wt{L}^2_t(\cB^{\f12})}\|w^1_\Th\|_{\wt{L}^2_{t,\dot\zeta(t)}(\cB^{1})}.
\end{split}
\eeno
Whereas thanks to \eqref{S3eq11},
we get
\beno
\begin{split}
\int_0^t&\bigl|\bigl(\D_k^{\rm h}(T^\h_{\p_yw^1}v)_\Th\ |\ \D_k^{\rm
h}w^1_\Th\bigr)_{L^2}\bigr|\,dt'\\
\lesssim & \sum_{|k'-k|\leq 4}\int_0^t\|S_{k'-1}^\h \p_yw^1_\Th(t')\|_{L^\infty_\h(L^2_{\rm v})}
\|\D_{k'}^\h v_\Th(t')\|_{L^2_\h(L^\infty_{\rm v})}\|\D_k^\h w^1_\Th(t')\|_{L^2}\,dt'\\
\lesssim & \sum_{|k'-k|\leq 4}d_{k'}\|S_{k'-1}^\h \p_yw^1_\Th\|_{L^2_t(L^\infty_\h(L^2_{\rm v}))}\|u_\Th\|^{\f12}_{L^\infty_t(\cB^{\f32})}
 \Bigl(\int_0^t\|\D_k^\h w^1_\Th(t')\|_{L^2}^2\|\p_yu_\Th(t')\|_{\cB^{\f12}}\,dt'\Bigr)^{\frac12}\\
\lesssim &d_k^22^{-k}\|u_\Th\|^{\f12}_{L^\infty_t(\cB^{\f32})}
\|\p_yw^1_\Th\|_{\wt{L}^2_t(\cB^{\f12})}\|w^1_\Th\|_{\wt{L}^2_{t,\dot\zeta(t)}(\cB^{1})}.
\end{split}
\eeno
Along the same line, we obtain
\beno
\begin{split}
\int_0^t\bigl|&\bigl(\D_k^{\rm h}(R^\h(v, \p_yw^1))_\Th\ |\ \D_k^{\rm
h}w^1_\Th\bigr)_{L^2}\bigr|\,dt'\\
\lesssim &2^{\f{k}2}\sum_{k'\geq k-3}\int_0^t\|{\D}_{k'}^\h v_\Th(t')\|_{L^\infty_\h(L^2_{\rm v})}\|\wt{\D}_{k'}^\h\p_y w^1_\Th(t')\|_{L^2}\|\D_k^\h w^1_\Th(t')\|_{L^2}\,dt'\\
\lesssim & 2^{\f{k}2}\sum_{k'\geq k-3}\|u_\Th\|^{\f12}_{L^\infty_t(\cB^{\f32})}
\|\wt{\D}_{k'}^\h\p_y w^1_\Th\|_{L^2_t(L^2)} \Bigl(\int_0^t\|\D_k^\h w^1_\Th(t')\|_{L^2}^2\|\p_yu_\Th(t')\|_{\cB^{\f12}}\,dt'\Bigr)^{\frac12}\\
\lesssim &d_k^22^{-k}\|u_\Th\|^{\f12}_{L^\infty_t(\cB^{\f32})}
\|\p_yw^1_\Th\|_{\wt{L}^2_t(\cB^{\f12})}\|w^1_\Th\|_{\wt{L}^2_{t,\dot\zeta(t)}(\cB^{1})}.
\end{split}
\eeno
As a consequence, we arrive at
\beq \label{S5eq12}
\begin{split}
\int_0^t\bigl|\bigl(\D_k^{\rm h}&({v}\p_yw^1)_\Th\ |\ \D_k^{\rm
h}w_\Th^1\bigr)_{L^2}\bigr|\,dt'
\lesssim d_k^22^{-k}\|u_\Th\|^{\f12}_{L^\infty_t(\cB^{\f32})}
\|\p_yw^1_\Th\|_{\wt{L}^2_t(\cB^{\f12})}\|w^1_\Th\|_{\wt{L}^2_{t,\dot\zeta(t)}(\cB^{1})}.
\end{split}
\eeq

\noindent$\bullet$\underline{
The estimate of $\int_0^t\big|\big(\Delta_k^\h(w^2\p_yu)_\Th | \Delta_k^\h w^1_\Th\big)_{L^2}\big|\,dt'$.}

By applying Bony's decomposition \eqref{Bony}   for the horizontal variable to $w^2\p_yu$, we write
\beno
w^2\p_y u=T^\h_{w^2}\p_yu+T^\h_{\p_yu}w^2+R^h(w^2,\p_yu).
\eeno
In view of  \eqref{S3eq12}, we have
\beno
\begin{split}
\Bigl(\int_0^t&\|S_{k'-1}^\h w^2_\Th(t')\|_{L^\infty}^2\|\p_yu_\Th(t')\|_{\cB^{\f12}}\,dt'\Bigr)^{\f12}
\lesssim d_{k'} 2^{\f{k'}2}\|w^1_\Th\|_{\wt{L}^2_{t,\dot{\zeta}(t)}(\cB^{1})},
\end{split}
\eeno
so that we get, by applying H\"older's inequality, that
\beno
\begin{split}
\int_0^t&\bigl|\bigl(\D_k^{\rm h}(T^\h_{w^2}\p_y u)_\Th\ |\ \D_k^{\rm
h}w_\Th^1\bigr)_{L^2}\bigr|\,dt'\\
\lesssim & \sum_{|k'-k|\leq 4}2^{-\f{k'}2}\int_0^t\|S_{k'-1}^\h w^2_\Th(t')\|_{L^\infty}\|\p_yu_\Th(t')\|_{\cB^{\f12}}\|\D_k^\h w^1_\Th(t')\|_{L^2}\,dt'\\
\lesssim & \sum_{|k'-k|\leq 4}2^{-\f{k'}2}\Bigl(\int_0^t\|S_{k'-1}^\h w^2_\Th(t')\|_{L^\infty}^2\|\p_yu_\Th(t')\|_{\cB^{\f12}}\,dt'\Bigr)^{\f12}\\
&\qquad\qquad\qquad\times\Bigl(\int_0^t\|\D_k^\h w^1_\Th(t')\|_{L^2}^2\|\p_yu_\Th(t')\|_{\cB^{\f12}}\,dt'\Bigr)^{\f12}\\
\lesssim & d_k^22^{-k}\|w^1_\Th\|_{\wt{L}^2_{t,\dot\zeta(t)}(\cB^{1})}^2.
\end{split}
\eeno
While thanks to \eqref{S3eq11}, we find
\beno
\begin{split}
\int_0^t&\bigl|\bigl(\D_k^{\rm h}(T^\h_{\p_yu}w^2)_\Th\ |\ \D_k^{\rm
h}w^1_\Th\bigr)_{L^2}\bigr|\,dt'\\
\lesssim & \sum_{|k'-k|\leq 4}\int_0^t\|S_{k'-1}^\h \p_yu_\Th(t')\|_{L^\infty_\h(L^2_{\rm v})}
\|\D_{k'}^\h w^2_\Th(t')\|_{L^2_\h(L^\infty_{\rm v})}\|\D_k^\h w^1_\Th(t')\|_{L^2}\,dt'\\
\lesssim & \sum_{|k'-k|\leq 4}2^{k'}\int_0^t\|\p_yu_\Th(t')\|_{\cB^{\f12}}\|\D_{k'}^\h w^1_\Th(t')\|_{L^2}\|\D_k^\h w^1_\Th(t')\|_{L^2}\,dt'\\
\lesssim &d_k^22^{-k}\|w^1_\Th\|_{\wt{L}^2_{t,\dot\zeta(t)}(\cB^{1})}^2.
\end{split}
\eeno
Along the same line, we obtain
\beno
\begin{split}
\int_0^t\bigl|&\bigl(\D_k^{\rm h}(R^\h(w^2, \p_y u))_\Th\ |\ \D_k^{\rm
h}w^1_\Th\bigr)_{L^2}\bigr|\,dt'\\
\lesssim &2^{\f{k}2}\sum_{k'\geq k-3}\int_0^t\|{\D}_{k'}^\h w^2_\Th(t')\|_{L^2_\h(L^\infty_{\rm v})}
\|\wt{\D}_{k'}^\h\p_y u_\Th(t')\|_{L^2}\|\D_k^\h w^1_\Th(t')\|_{L^2}\,dt'\\
\lesssim & 2^{\f{k}2}\sum_{k'\geq k-3}2^{\f{k'}2}\int_0^t\|{\D}_{k'}^\h w^1_\Th(t')\|_{L^2}\|\p_yu_\Th(t')\|_{\cB^{\f12}}\|\D_k^\h w^1_\Th(t')\|_{L^2}\,dt'\\
\lesssim &d_k^22^{-k}\|w^1_\Th\|_{\wt{L}^2_{t,\dot\zeta(t)}(\cB^{1})}^2.
\end{split}
\eeno
This gives rise to
\beq \label{S5eq13}
\begin{split}
\int_0^t\bigl|\bigl(\D_k^{\rm h}&(w^2\p_yu)_\Th\ |\ \D_k^{\rm
h}w_\Th^1\bigr)_{L^2}\bigr|\,dt'
\lesssim d_k^22^{-k}\|w^1_\Th\|_{\wt{L}^2_{t,\dot\zeta(t)}(\cB^{1})}^2.
\end{split}
\eeq

By summing up (\ref{S5eq9}-\ref{S5eq13}), we conclude the proof of \eqref{S5eq7}.
\end{proof}

\begin{proof}[Proof of \eqref{S5eq8}] We first observe from $\p_xu+\p_yv=0$ and Poincare inequality that
\beq \label{S5eq14}
\begin{split}
\e^2\int_0^t\big|\big(\Delta_k^\h (\p_tv)_{\Th}|\Delta_k^\h w^2_\Th\big)_{L^2}\big|dt'\lesssim&\e^2d_k^22^{-k}\|(\p_tu)_\Th\|_{\wt{L}^2_{t}(\cB^{\f32})}\|\p_yw^2_\Th\|_{\wt{L}^2_{t}(\cB^{\f12})},\\
\e^2\int_0^t\big|\big(\Delta_k^\h (\p_y^2v)_{\Th}|\Delta_k^\h w^2_\Th\big)_{L^2}\big|dt'\lesssim& \e^2d_k^22^{-k}\|\p_yu_\Th\|_{\wt{L}^2_{t}(\cB^{\f32})}\|\p_yw^2_\Th\|_{\wt{L}^2_{t}(\cB^{\f12})},\\
\e^4\int_0^t\big|\big(\Delta_k^\h (\p_x^2v)_{\Th}|\Delta_k^\h w^2_\Th\big)_{L^2}\big|dt'\lesssim& \e^4d_k^22^{-k}\|\p_yu_\Th\|_{\wt{L}^2_{t}(\cB^{\f52})}\|w^2_\Th\|_{\wt{L}^2_{t}(\cB^{\f32})}.
\end{split}
\eeq

\noindent$\bullet$\underline{
The estimate of $\int_0^t\big|\big(\Delta_k^\h(u^\e\p_xv^\e)_\Th | \Delta_k^\h w^1_\Th\big)_{L^2}\big|\,dt'$.}

We write $$ u^\e\p_xv^\e=u^\e\p_xw^2+u^\e\p_xv.$$
It follows from Lemma \ref{lem3.1} that
\beq\label{S5eq17}
\int_0^t\big|\big(\Delta_k^\h(u^\e\p_x w^2)_{\Th}|\Delta_k^\h w^2_\Th\big)_{L^2}\big|\,dt'\lesssim d_k^22^{-k} \|w^2_\Th\|_{\wt{L}^2_{t,\dot{\zeta}(t)}(\cB^{1})}^2.
\eeq
By applying Bony's decomposition for the horizontal variable to $u^\e\p_x v$  gives
\beno u^\e\p_x v=T^\h_{u^\e}\p_x v+T^\h_{\p_x v}{u^\e}+R^\h({u^\e},\p_xv).
\eeno
Due to $$ \|S_{k'-1}^\h u^\e_\Th(t')\|_{L^\infty}\lesssim \|u^\e_\Th(t')\|_{\cB^{\f12}}^{\f12}\|\p_y u^\e_\Th(t')\|_{\cB^{\f12}}^{\f12},$$
and \eqref{S3eq11}, we have
\beno
\begin{split}
\int_0^t&\bigl|\bigl(\D_k^{\rm h}(T^\h_{u^\e}\p_x v)_\Th\ |\ \D_k^{\rm
h}w_\Th^2\bigr)_{L^2}\bigr|\,dt'\\
\lesssim & \sum_{|k'-k|\leq 4}\int_0^t\|S_{k'-1}^\h u^\e_\Th(t')\|_{L^\infty}
\|\D_{k'}^\h\p_x v_\Th(t')\|_{L^2}\|\D_k^\h w^2_\Th(t')\|_{L^2}\,dt'\\
\lesssim & \sum_{|k'-k|\leq 4}2^{2{k'}}\|u^\e_\Th\|_{L^\infty_t(\cB^{\f12})}^{\f12}\|\D_k^\h u_\Th\|_{L^2_t(L^2)}
\Bigl(\int_0^t\|\p_y u^\e_\Th(t')\|_{\cB^{\f12}}\|\D_k^\h w^2_\Th(t')\|_{L^2}^2\,dt'\Bigr)^{\f12}\\
\lesssim & d_k^22^{-k}\|u^\e_\Th\|_{L^\infty_t(\cB^{\f12})}^{\f12}\|\p_y u_\Th\|_{\wt{L}_t^2(\cB^2)}\|w^1_\Th\|_{\wt{L}^2_{t,\dot\zeta(t)}(\cB^{1})}^2.
\end{split}
\eeno
While again thanks to \eqref{S3eq11}, we find
\beno
\|S_{k'-1}^\h\p_xv_\Th(t')\|_{L^\infty}\lesssim 2^{\f{k'}2}\|\p_y u_\Th(t')\|_{\cB^{2}},
\eeno
which leads to
\beno
\begin{split}
\int_0^t&\bigl|\bigl(\D_k^{\rm h}(T^\h_{\p_xv} u^\e)_\Th\ |\ \D_k^{\rm
h}w^2_\Th\bigr)_{L^2}\bigr|\,dt'\\
\lesssim & \sum_{|k'-k|\leq 4}\int_0^t\|S_{k'-1}^\h \p_xv_\Th(t')\|_{L^\infty}
\|\D_{k'}^\h u^\e_\Th(t')\|_{L^2}\|\D_k^\h w^2_\Th(t')\|_{L^2}\,dt'\\
\lesssim & \sum_{|k'-k|\leq 4}d_{k'}\|\p_yu_\Th\|_{\wt{L}^2_t(\cB^{2})}\|u^\e\|_{L^\infty_t(\cB^{\f12})}^{\f12}\Bigl(\int_0^t\|\p_y u^\e_\Th(t')\|_{\cB^{\f12}}\|\D_k^\h w^2_\Th(t')\|_{L^2}^2\,dt'\Bigr)^{\f12}\\
\lesssim & d_k^22^{-k}\|u^\e_\Th\|_{L^\infty_t(\cB^{\f12})}^{\f12}\|\p_y u_\Th\|_{\wt{L}_t^2(\cB^2)}\|w^1_\Th\|_{\wt{L}^2_{t,\dot\zeta(t)}(\cB^{1})}^2.
\end{split}
\eeno
Along the same line, we obtain
\beno
\begin{split}
\int_0^t\bigl|&\bigl(\D_k^{\rm h}(R^\h(u^\e, \p_xv))_\Th\ |\ \D_k^{\rm
h}w^2_\Th\bigr)_{L^2}\bigr|\,dt'\\
\lesssim &2^{\f{k}2}\sum_{k'\geq k-3}\int_0^t\|{\D}_{k'}^\h u^\e_\Th(t')\|_{L^2_\h(L^\infty_{\rm v})}
\|\wt{\D}_{k'}^\h\p_xv_\Th(t')\|_{L^2}\|\D_k^\h w^1_\Th(t')\|_{L^2}\,dt'\\
\lesssim & 2^{\f{k}2}\sum_{k'\geq k-3}2^{\f{3k'}2}\|u^\e_\Th\|_{L^\infty_t(\cB^{\f12})}^{\f12}\|\D_k^\h u_\Th\|_{L^2_t(L^2)}
\Bigl(\int_0^t\|\p_y u^\e_\Th(t')\|_{\cB^{\f12}}\|\D_k^\h w^2_\Th(t')\|_{L^2}^2\,dt'\Bigr)^{\f12}\\
\lesssim & d_k^22^{-k}\|u^\e_\Th\|_{L^\infty_t(\cB^{\f12})}^{\f12}\|\p_y u_\Th\|_{\wt{L}_t^2(\cB^2)}\|w^2_\Th\|_{\wt{L}^2_{t,\dot\zeta(t)}(\cB^{1})}.
\end{split}
\eeno
This gives rise to
\beq \label{S5eq18}
\begin{split}
\int_0^t\bigl|\bigl(\D_k^{\rm h}&(u^\e\p_xv)_\Th\ |\ \D_k^{\rm
h}w_\Th^2\bigr)_{L^2}\bigr|\,dt'
\lesssim d_k^22^{-k}\|u^\e_\Th\|_{L^\infty_t(\cB^{\f12})}^{\f12}\|\p_y u_\Th\|_{\wt{L}_t^2(\cB^2)}\|w^2_\Th\|_{\wt{L}^2_{t,\dot\zeta(t)}(\cB^{1})}.
\end{split}
\eeq

\noindent$\bullet$\underline{
The estimate of $\int_0^t\big|\big(\Delta_k^\h(v^\e\p_yv^\e)_\Th | \Delta_k^\h w^1_\Th\big)_{L^2}\big|\,dt'$.}

We first note that
\beno
v^\e\p_y v^\e=v\p_yw^2+w^2\p_yw^2+v\p_yv+w^2\p_yv.\eeno
We first deduce Lemma \ref{lem3.3} that
\begin{align*}
\e^2\int_0^t\bigl|\bigl(\D_k^{\rm h}(w^2\p_yw^2)_\Th\ |&\ \D_k^{\rm
h}w_\Th^2\bigr)_{L^2}\bigr|\,dt'
\lesssim d_k^22^{-k}
\bigl\|(w^1_\Th,\e w^2_\Th)\bigr\|_{\wt{L}^2_{t,\dot\zeta(t)}(\cB^{1})}^2.
\end{align*}
It follows from \eqref{S5eq12} that
\begin{align*}
\int_0^t\bigl|\bigl(\D_k^{\rm h}({v}\p_yw^2)_\Th\ |&\ \D_k^{\rm
h}w_\Th^2\bigr)_{L^2}\bigr|\,dt'
\lesssim d_k^22^{-k}
\|u_\Th\|^{\f12}_{L^\infty_t(\cB^{\f32})}
\|\p_yw^2_\Th\|_{\wt{L}^2_t(\cB^{\f12})}\|w^2_\Th\|_{\wt{L}^2_{t,\dot\zeta(t)}(\cB^{1})}.
\end{align*}
And \eqref{S5eq11} ensures that
\begin{align*}
\int_0^t\bigl|\bigl(\D_k^{\rm h}&({w^2}\p_xu)_\Th\ |\ \D_k^{\rm
h}w_\Th^2\bigr)_{L^2}\bigr|\,dt'\\
\lesssim &d_k^22^{-k}\|w_\Th^2\|_{\wt{L}^2_{t,\dot{\zeta}(t)}(\cB^{1})}\Bigl(\|w^2_\Th\|_{\wt{L}^2_{t,\dot{\zeta}(t)}(\cB^{1})}
+\|u_\Th\|^{\f12}_{L^\infty_t(\cB^{\f32})}
\|\p_yw^2_\Th\|_{\wt{L}^2_t(\cB^{\f12})}\Bigr).
\end{align*}
We deduce from the proof of \eqref{S5eq12} that
\begin{align*}
\int_0^t\bigl|\bigl(\D_k^{\rm h}({v}\p_y v)_\Th\ |\ \D_k^{\rm
h}w_\Th^2\bigr)_{L^2}\bigr|\,dt'
\lesssim & d_k^22^{-k}\|u_\Th\|^{\f12}_{L^\infty_t(\cB^{\f32})}
\|\p_yv_\Th\|_{\wt{L}^2_t(\cB^{\f12})}\|w^2_\Th\|_{\wt{L}^2_{t,\dot\zeta(t)}(\cB^{1})}\\
\lesssim & d_k^22^{-k}\|u_\Th\|^{\f12}_{L^\infty_t(\cB^{\f32})}
\|\p_y u_\Th\|_{\wt{L}^2_t(\cB^{\f32})}\|w^2_\Th\|_{\wt{L}^2_{t,\dot\zeta(t)}(\cB^{1})}.
\end{align*}
As a result, it comes out
\beq \label{S5eq21}
\begin{split}
\e^2\int_0^t\big|\big(&\Delta_k^\h(v^\e\p_yv^\e)_\Th | \Delta_k^\h w^1_\Th\big)_{L^2}\big|\,dt'
\lesssim d_k^22^{-k}\Bigl(
\bigl\|(w^1_\Th,\e w^2_\Th)\bigr\|_{\wt{L}^2_{t,\dot\zeta(t)}(\cB^{1})}^2 \\
&+\e^2\|u_\Th\|^{\f12}_{L^\infty_t(\cB^{\f32})}\bigl(
\|\p_yw^2_\Th\|_{\wt{L}^2_t(\cB^{\f12})}+\|\p_y u_\Th\|_{\wt{L}^2_t(\cB^{\f32})}\bigr)\|w^2_\Th\|_{\wt{L}^2_{t,\dot\zeta(t)}(\cB^{1})}\Bigr).
\end{split}
\eeq

Summing up (\ref{S5eq14}-\ref{S5eq21}) gives rise to \eqref{S5eq8}.
\end{proof}

\section*{Acknowledgments}

 Part of this work was done when Marius Paicu was visiting the Chinese Academy of Sciences and the Peking University  in June  2018. We appreciate the hospitality and the financial support of these institutions. M. Paicu was also partially supported by the Agence Nationale de la Recherche, Project IFSMACS, grant ANR-15-CE40-0010. P. Zhang is partially supported
by NSF of China under Grants   11371347 and 11688101, and innovation grant from National Center for
Mathematics and Interdisciplinary Sciences. Z. Zhang is partially supported by NSF of China under Grant 11425103.

\end{document}